\documentclass[12pt]{article}
\usepackage{amsmath,amsfonts,amsthm,amssymb,mathrsfs,mathtools}
\usepackage{graphicx}
\usepackage[usenames,dvipsnames]{color}
\usepackage{float}
\usepackage{graphicx}
\usepackage{subfigure}
\usepackage{caption}
\usepackage{blindtext}
\usepackage{cite}
\usepackage{url}
\usepackage{caption}
\usepackage{epstopdf}
\usepackage{hyperref}

\usepackage{color}
\usepackage[latin1]{inputenc}

\newcommand{\dint}{\displaystyle\int}

\theoremstyle{plain}
\newtheorem{theorem}{Theorem}[section]
\newtheorem{hy}{Assumption}[section]
\newtheorem{corollary}[theorem]{Corollary}
\newtheorem{lemma}[theorem]{Lemma}
\newtheorem{proposition}[theorem]{Proposition}

\theoremstyle{definition}
\newtheorem{definition}[theorem]{Definition}

\theoremstyle{remark}
\newtheorem{remark}[theorem]{Remark}

\numberwithin{equation}{section}
\numberwithin{theorem}{section}

\usepackage{a4wide}
\usepackage{geometry}
\begin{document}
\renewcommand{\thefootnote}{\fnsymbol{footnote}}

		\begin{center}
			{\Large \textbf{Lévy bridges with random length }} \\[0pt]
			~\\[0pt] \textbf{Mohamed Erraoui} \footnote[1]{Mathematics Department, Faculty of Sciences Semalalia, Cadi Ayyad University, Boulevard Prince Moulay Abdellah,	P. O. Box 2390, Marrakesh 40000, Morocco. 
				E-mail: \texttt{erraoui@uca.ac.ma}} \textbf{Astrid Hilbert} \footnote[2]{Linnaeus University, Vejdesplats 7, SE-351 95 V\"axj\"o, Sweden.
				 E-mail: \texttt{astrid.hilbert@lnu.se}} and \textbf{Mohammed Louriki} \footnote[7]{Mathematics Department, Faculty of Sciences Semalalia, Cadi Ayyad University, Boulevard Prince Moulay Abdellah,	P. O. Box 2390, Marrakesh 40000, Morocco. 
				 E-mail: \texttt{louriki.1992@gmail.com}}
			\\[0pt]
			\textit{Cadi Ayyad University and Linnaeus University  }\\[0pt]
			~\\[0pt]
		\end{center}
\begin{abstract}
	In this paper our first goal is to give precise definition of the L\'evy bridges with random length.  Our second task is to establish the Markov property of this process with respect to its completed natural filtration and thus with respect to the usual augmentation of this one. This property will be crucial for the right-continuity of completed natural filtration.
\end{abstract}


\begin{center}
\section{Introduction}
\label{Setion_1}
\end{center}

A bridge with deterministic length is a stochastic process obtained by conditioning a known process to start from an initial point at time $0$ and to arrive to some fixed point $z$ at a fixed future time $r>0$. For example Brownian, Gamma, Gaussian, L\'evy and Markov bridges, see \cite{EY}, \cite{FG}, \cite{FPY}, \cite{GSV} and \cite{HHM}. It should be mentioned that Brownian, Gamma and Gaussian bridges have simple explicit pathwise constructions whereas for general Lévy and Markov bridges special considerations are needed. This stems from the specific characteristics of their laws. We refer to the paper \cite{FPY} for a general discussion on the construction of bridges for Markov process.

Recently, using pathwise representations of bridges with deterministic length, Bedini et al. \cite{BBE} and Erraoui et al. \cite{EL} and \cite{EHL} introduced the Brownian, Gaussian and Gamma bridges with random length using randomization approach. That is to substitute the deterministic time $r$ with  the values of a random time $\tau$.
A natural question arise: How to define the Lévy bridge with random length $\tau$? 

 In this paper we give the definition of the Lévy bridge with a random length $\tau$. 
 Our approach to do is as follows: conditionally on the event $\tau=r$, the law of the Lévy bridge with a random length $\tau$ is none other than that of the Lévy bridge with deterministic length $r$.  This will allow us to construct the law of Lévy bridge with a random length $\tau$ by integrating the above conditional law with respect to the law of $\tau$.
 
For this, following the approach given in \cite{FPY} for the construction of a deterministic length bridge $r$, we will need the transition probability densities of the Lévy process. Moreover, for reasons of integrability, we should know the asymptotic behavior of these probability densities. This led us to suppose the integrability condition of the characteristic function, introduced by Sharpe in \cite{Sh} and used in several works, see for example \cite{CUB}. It should be noted that under Sharpe's condition we have both the existence and the asymptotic behavior of the probability densities, see \cite{KS}.

Once the construction is done, our investigation focuses on the Markov property with respect to both filtrations, namely the completed natural filtration of the Lévy bridge and it's usual augmentation. The latter is the smallest filtration containing the natural filtration and satisfying the usual hypotheses of right-continuity and completeness. In order to accomplish the goal, we shall impose the condition that, away from the origin, the densities remain uniformly bounded as $t$ goes to $0$. We recall that this kind of condition was proved in several works, see \cite{Ish} and \cite{Lea87} and recently used in \cite{FH} and \cite{Kol} to derive a small-time polynomial expansion for the transition density of a Lévy process. As an important consequence of the Markov property we derive the right continuity of the the completed natural filtration, which means that both filtrations coincide.

In Section 2, we present a brief introduction to L\'evy processes, a note on the existence of transition probability densities for L\'evy processes, and some useful properties of the L\'evy bridge with deterministic length $r$, which will be used throughout the paper. In section 3, we first define the bridge with random length $\tau$ and we consider the stopping time property of $\tau$ with respect to the completed natural filtration of the former. We also clarify that the bridge with random length inherits the Markov property from its associated bridge with deterministic length. In the last section we will give some examples, more precisely classes of Lévy processes whose characteristic exponent behaves as in the symmetric stable Lévy case.

The following notations will be used throughout the paper: 
For a complete probability space $(\Omega,\mathcal{F},\mathbb{P})$, $\mathcal{N}_p$ denotes the
collection of $\mathbb{P}$-null sets. If $\theta$ (resp. $X$) is a random variable (resp. stochastic process), then $\mathbb{P}_{\theta}$
(resp. $\mathbb{P}_{X}$) denotes the law of $\theta$ (resp. of $X$) under $\mathbb{P}$. If $E$ is a topological space, then the Borel $\sigma$-algebra over $E$ will be denoted by $\mathcal{B}(E)$. The characteristic function of a set $A$ is written $\mathbb{I}_{A}$. 
The symmetric difference of two sets $A$ and $B$ is denoted by $A\Delta B$. The Skorohod space $\mathbb{D}_{\infty}$ of right continuous functions with left limits (càdlàg) from $[0,+\infty[$
to $\mathbb{R}$.
Finally for any process $Y=(Y_t,\, t\geq 0)$ on $(\Omega,\mathcal{F},\mathbb{P})$, we define by:
\begin{enumerate}
	\item[(i)] $\mathbb{F}^{Y}=\bigg(\mathcal{F}^{Y}_t:=\sigma(Y_s, s\leq t),~ t\geq 0\bigg)$ the natural filtration of the process $Y$.
	\item[(ii)] $\mathbb{F}^{Y,c}=\bigg(\mathcal{F}^{Y,c}_t:=\mathcal{F}^{Y}_t\vee \mathcal{N}_{P},\, t\geq 0\bigg)$ the completed natural filtration of the process $Y$.
	\item[(iii)] $\mathbb{F}^{Y,c}_{+}=\bigg(\mathcal{F}^{Y,c}_{t^{+}}:=\underset{{s>t}}\bigcap\mathcal{F}^{Y,c}_{s}=\mathcal{F}^{Y}_{t^{+}}\vee \mathcal{N}_{P},\, t\geq 0\bigg)$ the smallest filtration containing $\mathbb{F}^{Y}$ and satisfying the usual hypotheses of right-continuity and completeness.
\end{enumerate}

\begin{center}
	\section{L\'evy bridges with deterministic length}
\end{center}
This part summarizes a few well-known results about one-dimensional
L\'evy processes and the existence of its probability densities. Further details can be found in Bertoin
\cite{B} and Sato \cite{Sato}. 

\subsection{Absolute continuity of Lévy processes}
A real-valued stochastic process $X=\left\lbrace X_t, t\geq 0\right\rbrace$ defined on $\left(\Omega,\mathcal{F},\mathbb{P}\right)$ is said to be a Lévy process if it possesses the following properties: 

(i) The paths of $X$ are $\mathbb{P}$-almost surely right continuous with left limits. 

(ii) $\mathbb{P}(X_{0}=0)=1$. 

(iii) For $0\leq s\leq t$, $X_{t}-X_{s}$ is equal in distribution to $X_{t-s}$. 

(iv) For $0\leq s\leq t$, $X_{t}-X_{s}$ is independent of $\left\{ X_{u}:u\leq s\right\} $.

\vspace{0,2cm}

The law of $X_t$ is specified via its characteristic function given by
$$ \mathbb{E}[\exp(i u X_t)]=\exp(t\psi(u)), \,\,u \in \mathbb{R}, t\geq 0. $$
The function $\psi$ is known as the characteristic exponent of the process $X$.  The explicit
form of the characteristic exponent is given through the well-known L\'evy-Khintchine formula:
\begin{equation}
\psi(u)=iu d-\dfrac{u^2 b}{2}+\int_{\mathbb{R}}\left(\exp(iu x)-1-iu x \mathbb{I}_{\{\vert x\vert<1\}}\right)\nu(dx), \,\,\label{eqLevykinchen}
\end{equation} 
where $d \in \mathbb{R},$ $b \in \mathbb{R}_+$ and $\nu$ is a measure concentrated on $\mathbb{R}^{*}$, called the L\'evy measure, satisfying $$\int_{\mathbb{R}}(x^2\wedge 1)\nu(dx) < \infty.$$
The Kolmogorov-Daniell theorem allows us to see that the finite-dimensional
distributions of $X$ induce a probability measure $\mathbb{P}_X$ on the Skorohod space $\mathbb{D}_{\infty}$. On the other hand, it is well-known that any Lévy process can be realized
as the coordinate process on the Skorohod space $\mathbb{D}_{\infty}$ as follows:

 $Z_t(w)=w(t)$ for $w\in\mathbb{D}_{\infty} $ the canonical
process of the coordinates, $\mathcal{G}:=\sigma\left(Z_{s},s\geq0\right)$
with $\mathcal{G}_{t}:=\sigma\left(Z_{u},0\leq u\leq t\right),\, t\geq 0$, for
the canonical filtration generated by $Z$, and $\mathbb{P}_X$ is the probability measure on $\left(D_{\infty},\mathcal{G} \right)$. Then $Z$ is a Lévy process with the same distribution as $X$ on $\left(\Omega,\mathcal{F},\mathbb{P}\right)$. So we will regard each Lévy process as a probability measure on the
Skorokhod space $\mathbb{D}_{\infty}$ and vice versa. 

Many papers are devoted to the sufficient conditions under
which the probability law $\mathbb{P}_{X_{t}}(dx)$ is absolutely continuous with respect to Lebesgue measure, see Sato \cite{Sato}, Sharpe \cite{Sh}, Tucker \cite{T}, and Hartman and Wintner \cite{HW}. 
From now on, we suppose that the process $X$ is a symmetric Lévy process satisfying Sharpe's condition: 
\begin{hy}\label{hyintegr}
		 $\exp(t\psi)$ is integrable for any $t>0$.
	\end{hy}  
We note that the characteristic exponent $\psi$ is real-valued since $X$ is symmetric. Under Assumption \ref{hyintegr} for each $t>0$ the law 
$\mathbb{\ensuremath{P}}_{X_{t}}$ is known to be absolutely continuous with a bounded jointly continuous density $f_{t}(.)$ on $(0,\infty)\times\mathbb{R}$ given by 
\begin{equation}
			f_t(x)= (2\pi)^{-1}\dint_{\mathbb{R}}\exp(ixy)\exp(t\psi(y))\,dy, \,\, x\in \mathbb{R}.\label{eqfrepresentation}
		\end{equation}
		 Furthermore, by independence and homogeneity of the increments, the transition probabilities of $X$ can be taken
\begin{equation}
\mathbb{P}_{t}(x,dy)=f_t(y-x)dy.\label{eqtransitiondensity}
\end{equation}
Therefore the Chapman-Kolmogorov identity
\begin{equation}
f_{s+t}(x)=\int_{-\infty}^{+\infty}f_{t}(x-y)f_s(y)dy\label{eqChapmanKolmogorov}
\end{equation}
holds for every $s, t>0$ and $x, y\in \mathbb{R}$. In this case, given $x_0=t_0=0$, the finite-dimensional distributions are given by
\begin{equation}
\mathbb{P}(X_{t_1}\in dx_1,\ldots,X_{t_n}\in dx_n )=\prod_{i=1}^{n} f_{t_i-t_{i-1}}(x_i-x_{i-1})\,dx_i,
\end{equation}
for every $n\in \mathbb{N}$, $0<t_1<t_2<...<t_n$, and $(x_1,x_2,...,x_n) \in \mathbb{R}^n$.

Unfortunately, neither the density function $f_t$ nor its distribution function $\mathbb{P} (X_t \leq y)$ are explicitly given in many cases and therefore it is natural that the asymptotic behaviour of $f_t$ is given in terms of the characteristic exponent $\psi$. This result is established first in \cite{Sc} for one dimension and in \cite{KS} for the $n$-dimensional case. Before the statement we need some notations. For each $x\in \mathbb{R}$, we shall denote by $\mathbb{P}_x$ (resp., $\mathbb{E}_x$), the measure corresponding to the process $\left\lbrace x+ X_t, t\geq 0\right\rbrace$ under $\mathbb{P}$ (resp., the expectation under  $\mathbb{P}_x$). Also let us note that	under Assumption \ref{hyintegr} we have
		\begin{equation}
	0<f_t(0)=(2\pi)^{-1}||\exp(t\psi)||_{L^1(\mathbb{R})}<\infty,\,\, \text{for all}\,\, t>0.\label{fzero}
	\end{equation}
			\begin{proposition}\label{propschilling}
The following limits exist locally uniformly for all  $x\in \mathbb{R}$ :
				\begin{alignat}{2}
				\lim\limits_{t\longrightarrow \infty}\dfrac{\mathbb{E}_x[g(X_t)]}{||\exp(t\psi)||_{L^1}}&=(2\pi)^{-1}\dint_{\mathbb{R}}g(z)dz,\,\text{for all}\, g\in L^{1}(\mathbb{R}),\label{eqschilling1} \\  \notag \\
				\lim\limits_{t\longrightarrow \infty}\dfrac{f_t(x)}{f_t(0)}&=1.\label{eqschilling3}
				\end{alignat}
				
			\end{proposition}
		The following corollary is needed in very few places in the sequel.	
			\begin{corollary}\label{corschilling}
The following limits exist locally uniformly for all  $x\in \mathbb{R}$ :
				\begin{alignat}{2}
					(i) \lim\limits_{r\longrightarrow \infty}\dfrac{f_{r-t}(x)}{f_r(0)}&=1,\,\,\text{for all } t\geq 0,\label{eqrmkschilling} \\ \notag\\
(ii) \lim\limits _{r\longrightarrow\infty}\dfrac{f_{r-t}(z-x)}{f_{r}(z)}&=1,\,\,\text{for all } t\geq 0,	\label{asybeh}				
					\end{alignat}
				
for any $z\in\mathbb{R}$ such that $0<f_t(z)<\infty,\, \text{for all }\, t>0$.
				\end{corollary}
				\begin{proof}
In order to prove \eqref{eqrmkschilling} it is sufficient, with the help of \eqref{eqschilling3}, to show that
 \begin{equation}
						\lim\limits_{r\longrightarrow \infty}\dfrac{f_{r-t}(0)}{f_r(0)}=1,\label{eqschillingapp}
						\end{equation}
						which is equivalent to 
						\begin{equation}
						\lim\limits_{r\longrightarrow \infty}\dfrac{f_{r+t}(0)}{f_r(0)}=1.\label{eqvschillingapp}
						\end{equation}
						Now the Chapman-Kolmogorov identity \eqref{eqChapmanKolmogorov} gives 
						\begin{equation*}
						f_{r+t}(0)=\mathbb{E}[f_{t}(-X_r)].
						\end{equation*}
						Using \eqref{fzero}, \eqref{eqschilling1} and the fact that $f_t$ is a probability density function we obtain \eqref{eqrmkschilling}. By the same way we get \eqref{asybeh}.
						\end{proof}
						\subsection{L\'evy bridges of deterministic length}
Let $r>0$ and $z\in \mathbb{R}$ such that 
		\begin{equation}
		0<f_r(z)<\infty.\label{equationf_tcondition}
		\end{equation}
		We emphasize that this condition is satisfied for $z$ in the interior of the support of the law $\mathbb{P}_{X_{t}}$, cf Sharpe \cite{Sh}.
		
		A Lévy bridges from $0$ to $z$ of deterministic length $r>0$ is
a Lévy process conditioned to start at $0$ end at $z$ at the given
time $r$, provided that such a process exists. Markovian bridges, in particular Lévy bridges, were first
constructed in \cite{Kal81} using the convergence criteria for processes with exchangeable increments of \cite{Kal73}, then afterwards in \cite{FPY} under duality hypothesis and recently in \cite{CUB} under the Assumption
$\ref{hyintegr}$. 

Since $\mathbb{\ensuremath{P}}_{X_{t}}(dx)=f_{t}(x)\,dx$ then using
Tonelli's theorem we obtain
\[
\int_{\mathbb{R}}\,h(x)\,\mathbb{E}_{x}\left(g(X_{t})\right)\,dx=\int_{\mathbb{R}}\,g(x)\,\mathbb{E}_{x}\left(h(-X_{t})\right)\,dx,
\]
for all $t>0$ and all positive measurable functions $g$ and $h$.
This means that there is a second right process $\hat{X}$ with càdlàg
paths in duality with $X$ relative to the Lebesgue measure. In other
words, the semigroups $(P_{t})_{t\ge0}$ of $X$, $P_{t}(g)(x):=\mathbb{E}_{x}\left(g(X_{t})\right)$,
and $(\hat{P}_{t})_{t\ge0}$ of $\hat{X}$, $\hat{P}_{t}(h)(x):=\mathbb{E}_{x}\left(h(-X_{t})\right)$,
are related by 
\begin{equation}
\int_{\mathbb{R}}\,h(x)\,P_{t}(g)(x)\,dx=\int_{\mathbb{R}}\,g(x)\,\hat{P}_{t}(h)(x)\,dx.\label{duality}
\end{equation}
Moreover, it is easy to see that the transition probabilities of $\hat{X}$ is given by 
\[
\mathbb{\hat{P}}_{t}(x,dy)=f_{t}(x-y)dy.
\]
See \cite{FPY}, \cite{GS} and \cite{W}  for a detailed account.

Now putting \eqref{eqtransitiondensity}, \eqref{eqChapmanKolmogorov} and \eqref{duality} together we may apply Proposition 1 in \cite{FPY} to conclude that there exists a Lévy bridge, $X^{r,z}= \left\lbrace X_t^{r,z}, 0\leq t\leq r \right\rbrace$, from $0$ to $z$ of length $r>0$ associated with the L\'evy process $X$. The process $X^{r,z}$ thus constructed can be realized
as the coordinate process on the Skorohod space $\mathbb{D}_{r}$ of càdlàg functions
from $[0,r]$ to $\mathbb{R}$. 
Furthermore, it is a non-homogeneous strong Markov process with transition
		densities given by
		\begin{align}
		\mathbb{P}(X_u^{r,z}\in dy\vert X_t^{r,z}=x)&=\dfrac{f_{u-t}(y-x)f_{r-u}(z-y)}{f_{r-t}(z-x)}dy, \,\,\,\,0\leq t<u<r.\label{eqtransitionlawofX^{r,z}}
		\end{align}
		Consequently, the finite-dimensional densities of $\left\lbrace X^{r,z},0<t<r \right\rbrace$ exist and, given $x_0=t_0=0$, for every $n\in \mathbb{N}$, $0<t_1<t_2<...<t_n<r$, and $(x_1,x_2,...,x_n) \in \mathbb{R}^n$, we have
		\begin{align*}
		\mathbb{P}(X^{r,z}_{t_1}\in dx_1,\ldots,X^{r,z}_{t_n}\in dx_n )&=\varphi_{X_{t_{1}}^{r,z},\ldots,X_{t_{n}}^{r,z}}\left(x_{1},\ldots,x_{n}\right)\,dx_1\ldots dx_n,
		\end{align*}
			where 
			\begin{equation}
			\varphi_{X_{t_{1}}^{r,z},\ldots,X_{t_{n}}^{r,z}}\left(x_{1},\ldots,x_{n}\right)=\dfrac{f_{r-t_n}(z-x_n)}{f_r(z)}\prod_{i=1}^{n}f_{t_i-t_{i-1}}(x_i-x_{i-1}).\label{equationfinidimensionaldensity}
			\end{equation}
\begin{remark}\label{ratio}
	(i) The condition \eqref{equationf_tcondition} is required to ensure that the ratio in \eqref{eqtransitionlawofX^{r,z}} is well defined, i.e., \eqref{equationf_tcondition} is suffices to ensure that 
	$$ y\longrightarrow \dfrac{f_{u-t}(y-X_t^{r,z})f_{r-u}(z-y)}{f_{r-t}(z-X_t^{r,z})} $$
	is a well defined probability density function for almost every value of $X_t^{r,z}$. Indeed, using the Chapman-Kolmogorov identity \eqref{eqChapmanKolmogorov}, one can show
	\[
\begin{array}{lll}
\mathbb{E}\left(\dint_{\mathbb{R}}\dfrac{f_{u-t}(y-X_{t}^{r,z})f_{r-u}(z-y)}{f_{r-t}(z-X_{t}^{r,z})}\,dy\right) & = & \dint_{\mathbb{R}}\int_{\mathbb{R}}\,\dfrac{f_{u-t}(y-x)\,f_{r-u}(z-y)}{f_{r-t}(z-x)}\,\mathbb{P}(X_{t}^{r,z}\in dx)\,dy\\
\\
 & = & \dint_{\mathbb{R}}\int_{\mathbb{R}}\,\dfrac{f_{u-t}(y-x)\,f_{r-u}(z-y)}{f_{r-t}(z-x)}\,\dfrac{f_{t}(x)f_{r-t}(z-x)}{f_{r}(z)}\,dx\,dy\\
\\
 & = & \dint_{\mathbb{R}}\,\dfrac{f_{r-u}(z-y)}{f_{r}(z)}\int_{\mathbb{R}}f_{u-t}(y-x)\,f_{t}(x)\,dx\,dy\\
\\
 & = & \dfrac{1}{f_{r}(z)}\dint_{\mathbb{R}}\,f_{r-u}(z-y)\,f_{u}(y)\,dy=1.
\end{array}
\]

	Thus $$ \mathbb{P}\bigg(\dint_{\mathbb{R}}\,\dfrac{f_{u-t}(y-X_t^{r,z})f_{r-u}(z-y)}{f_{r-t}(z-X_t^{r,z})}dy=1\bigg)=1. $$

(ii) In all that follows the process  $X^{r,z}$ is continued beyond time $r$ with the constant value $z$. We thereby identify $ \left\lbrace X_t^{r,z}, t\geq 0 \right\rbrace$ with the process $\tilde{X}^{r,z}:=\left\lbrace  X_t^{r,z}\mathbb{I}_{\{t<r\}}+z\mathbb{I}_{\{t\geq r\}},\, t\geq 0\right\rbrace$.

\end{remark}

\begin{center}
	\section{L\'evy bridges with random length}\label{sectionstoppingtimeproperty}
\end{center}
In this section our first task is to define the Lévy bridge $\zeta^z$ from $0$ to $z$ of random length $\tau$.  Then we prove that the random time $\tau$ is a stopping time with respect to $\mathbb{F}^{\zeta^z,c}$ and we give the regular conditional distribution of $\tau$ and $(\tau,\zeta_.^z)$ given the n-coordinate of $\zeta^z$. Thereafter we show that the process $\zeta^z$ is a non-homogeneous Markov process with respect to its completed natural filtration $\mathbb{F}^{\zeta^z,c}$. Finally, under additional conditions, we prove that the filtration $\mathbb{F}^{\zeta^z,c}$ is right-continuous.\\		

\begin{definition}\label{defzeta^z}
		Let $\tau: (\Omega,\mathcal{F},\mathbb{P}) \longmapsto (0,+\infty)$ be a strictly positive random time, with distribution function $F(t) := \mathbb{P}(\tau \leq t)$, $t \geq 0$. We say that a process $\zeta^z$ is the bridge  from $0$ to $z$ of random length $\tau$ derived from the Lévy process $X$ if the following are satisfied:
		\begin{enumerate}
			\item[(i)] $0<f_r(z)<\infty$ for $\mathbb{P}_{\tau}$ almost every $r$.
			\item[(ii)] The conditional distribution of $\zeta^z$ given $\left\{ \tau=r\right\} $ is the law of the process $X^{r,z}$.
		\end{enumerate}
		
			We denote by $\mathcal{Z}$ the set of $z$ satisfying (i).
		\end{definition}
		\begin{remark}
(i) It should be noted, under Assumption \ref{hyintegr}, that $0<f_t(0)<\infty$ for all $t>0$, see \eqref{fzero}. Moreover, once again under Assumption \ref{hyintegr}, it follows from \cite{Sh} that there exist a constant $c$ such that the set $\left\lbrace (t,z)\in (0,+\infty)\times \mathbb{R}:f_{t}(z)>0\right\rbrace$ is contained in the wedge $\left\lbrace (t,z):z>ct\right\rbrace$. This entails that $c \in [-\infty, 0[$ and $[0,+\infty[$ is contained  in $\mathcal{Z}$.

On the other hand, taking into account that $X$ satisfies assumption \ref{hyintegr}, if we assume in addition that $X$ is not a compound Poisson process and for all $c\geq 0$, the process $\left\lbrace\vert X_{t}-ct\vert, t\geq 0 \right\rbrace$ is not a subordinator then it follows from \cite{Sh} that $$0<f_t(z)<\infty, \, \text{for all}\, t > 0\,\, \text{and}\, \, z \in  \mathbb{R}.$$In this case $\mathcal{Z}$ is the whole $\mathbb{R}$.

(ii) The process $\zeta^z$ can be realized as follows : Consider the probability space $\left(\tilde{\Omega},\tilde{\mathcal{F}},\tilde{\mathbb{P}}\right)$
\[
\tilde{\Omega}=\mathbb{D}_{\infty}\times]0,+\infty[,\quad\tilde{\mathcal{F}}=\mathcal{F}\otimes\mathcal{B}\left(]0,+\infty[\right)\quad\tilde{\mathbb{P}}\left(dw,dr\right)=\mathbb{P}_{X^{r,z}}(dw)\,\mathbb{P}_{\tau}(dr).
\]

We write $\tilde{w}=(w,r)$ for the generic point of $\tilde{\Omega}$.
Now define 
\[
\tilde{\tau}(\tilde{w})=r\quad\text{and \quad}\tilde{\zeta}_{t}^{z}(\tilde{w})=X_{t}^{r,z}(w),\quad t\geq0.
\]
Thus $$ \tilde{\mathbb{P}}\left(\tilde{\tau}\leq t \right) =\mathbb{P}(\tau \leq t),\,\, t\geq 0,$$ and the conditional distribution of $\tilde{\zeta}^{z}$ given $\left\{ \tilde{\tau}=r\right\} $
is the law $\mathbb{P}_{X^{r,z}}$.  Moreover the process thus defined has a càdlàg paths and  satisfies $\tilde{\zeta}_{t}^{z}=z$ when $\tilde{\tau} \leq t$.
\end{remark}
\subsection{Weak continuity, stopping time property and conditional laws}	  
We begin with the weak continuity of the the L\'evy  bridge $\zeta^z$ with random length $\tau$ with respect to its ending point $z$.
	\begin{theorem}
		The map $z\longrightarrow \mathbb{P}_{\zeta^z}$ is weakly continuous in the sense that  the laws $\mathbb{P}_{\zeta^z}$ are weakly continuous on $\mathbb{D}_{\infty}$ with respect to $z\in \mathcal{Z}$.
	\end{theorem}
	\begin{proof}
	 First recall that the bridge laws $\mathbb{P}_{ X^{r,z}}$ on $\mathbb{D}_{r}$ are weakly  continuous with respect to $z$. This follows at once from Assumption \ref{hyintegr}, see Corollary 1 in \cite{CUB}. Hence, if $K$ is a compact set in $\mathcal{Z}$, it is easy to see that the family of laws $\left(\mathbb{P}_{ \tilde{X}^{r,z}}\right)_{ z\in K}$ on $\mathbb{D}_{\infty}$ are tight. So in order to establish the weak convergence of $\mathbb{P}_{ \tilde{X}^{r,z}}$ to $\mathbb{P}_{ \tilde{X}^{r,\tilde{z}}}$ as $z$ tends to $\tilde{z}\in \mathcal{Z}$, we need to prove the finite dimensional distributions convergence which is a simple consequence of Slutsky's theorem. Now making the identification $X^{r,\tilde{z}}=\tilde{X}^{r,\tilde{z}}$ we have
		\begin{align*}
			\lim\limits_{z\longrightarrow \tilde{z}}\mathbb{E}[F(\zeta^z_t, t\geq 0)]&=\lim\limits_{z\longrightarrow \tilde{z}}\dint_{(0,\infty)}\mathbb{E}\left[F(\zeta^z_t, t\geq 0)|\tau=r\right]\,\mathbb{P}_{\tau}(dr)\\
			&=\lim\limits_{z\longrightarrow \tilde{z}}\dint_{(0,\infty)}\mathbb{E}\left[F(X_t^{r,z}, t\geq 0)\right]\,\mathbb{P}_{\tau}(dr)\\
			&=\dint_{(0,\infty)}\mathbb{E}[(X_t^{r,\tilde{z}}, t\geq 0)]\, \mathbb{P}_{\tau}(dr)\\
			&=\mathbb{E}[F(\zeta^{\tilde{z}}_t, t\geq 0)],
		\end{align*}
		for every bounded continuous functional $F$ on $\mathbb{D}_{\infty}$ which ends the proof.
	\end{proof} 
	In the next result we study the stopping time property of the random time $\tau$ with respect to the completed natural filtration of $\zeta^z$.
	\begin{proposition}
        For all $t>0$, we have 
        \begin{equation}\label{stpprop}
        \mathbb{P}\left(\{\zeta^z_t = z\} \bigtriangleup \{\tau \leq t\}\right)=0.
        \end{equation}
		Then $\tau $ is a stopping time with respect to $\mathbb{F}^{\zeta^z,c}$ and consequently it is a stopping time with respect to $\mathbb{F}^{\zeta^z,c}_{+}$.
	\end{proposition}
	\begin{proof}
		First we have $\{\tau\leq t\}\subseteq \{\zeta^z_t=z\}$. On the other hand, using the formula of total probability and the assertion (ii) of Definition \ref{defzeta^z}, we obtain
		\begin{align*}
		\mathbb{P}(\zeta^z_{t}=z,t<\tau)&=\dint_{(t,+\infty)} \mathbb{P}(\zeta^z_{t}=z | \tau=r) \mathbb{P}_{\tau}(dr)
		\\  
		&=\dint_{(t,+\infty)} \mathbb{P}(X_{t}^{r,z}=z) \mathbb{P}_{\tau}(dr) \\
		&=0.
		\end{align*}
	    The latter equality uses the fact that the distribution of $X_t^{r,z}$ is absolutely continuous with respect to Lebesgue measure for all $r>t>0$. Thus $\mathbb{P}\left(\{\zeta_t^z = z\} \bigtriangleup \{\tau \leq t\}\right)=0$.
		It follows that the event $\{\tau \leq t\}$ belongs to  $\mathcal{F}_t^{\zeta^z} \vee \mathcal{N}_P,$ for all $t \geq 0$. Hence $\tau$ is a stopping time with respect to $\mathbb{F}^{\zeta^z,c}$ and consequently it is also a stopping time with respect to $\mathbb{F}^{\zeta^z,c}_{+}$.
	\end{proof}
	
	In order to determine the conditional law of the random time $\tau$ given $(\zeta^z_{t_1},\ldots,\zeta^z_{t_n})$ we will use the following
	
	\begin{proposition}\label{propbayesestimatejusquatn}
		Let $n\in \mathbb{N}^*$ and $0=t_0<t_1<t_2<...<t_n$ such that $F(t_1)>0$.  Let $g:\mathbb{R}_{+}\longrightarrow \mathbb{R} $ be a Borel function satisfying $\mathbb{E}[|g(\tau)|]<+\infty$. Then, $\mathbb{P}$-a.s., we have

		\begin{align}\label{taucondxi}
		\mathbb{E}[g(\tau)\vert\zeta^z_{t_{1}},\ldots,\zeta^z_{t_{n}}]=&  \dint_{(0,t_{1}]}\frac{g(r)}{F(t_{1})}\mathbb{P}_{\tau}(dr)\;\mathbb{I}_{\{\zeta_{t_{1}}=z\}} \nonumber \\ \nonumber
		\\
		 +&\sum\limits _{k=1}^{n-1}\dint_{(t_{k},t_{k+1}]}g(r)\Upsilon_{k}(r,\zeta^z_{t_{k}}) 
		\mathbb{P}_{\tau}(dr)\; \mathbb{I}_{\{\zeta^z_{t_{k}}\neq z,\zeta^z_{t_{k+1}}=z\}} \nonumber  \\ \nonumber
		\\
		+&\dint_{(t_{n},+\infty)}g(r)\Upsilon_{n}(r,\zeta^z_{t_{n}})\mathbb{P}_{\tau}(dr)\; \mathbb{I}_{\{\zeta^z_{t_{n}}\neq z\}}.
		\end{align}
		Here the functions $\Upsilon_{k}$ and $\Upsilon_{n}$ are defined by
		
		\begin{equation}
		\Upsilon_{k}(r,x):=\dfrac{\phi(r,t_k,x)}{\dint_{(t_{k},t_{k+1}]}\phi(r,t_k,x)\mathbb{P}_{\tau}(dr)},\; r\in (t_{k},t_{k+1}], \,\,  x\in \mathbb{R},\, \,k=1,\ldots,n-1,
		\end{equation}
		and
		\begin{equation}
		\Upsilon_{n}(r,x):=\dfrac{\phi(r,t_{n},x)}{\dint_{(t_{n},+\infty)}\phi(r,t_{n},x)\mathbb{P}_{\tau}(dr)}, ~~r\in (t_{n},+\infty), \,\,x\in \mathbb{R}, \label{equationphi}
		\end{equation}
		where 
		\begin{equation}
			\phi(r,t,x):=\dfrac{f_{r-t}(z-x)}{f_r(z)}\mathbb{I}_{\{t<r\}},\,\, r>0,\,\, t>0,\,\, x\in \mathbb{R}. \label{eqphi}
		\end{equation}
		
	\end{proposition}
		\begin{proof} It follows from the Definition \eqref{defzeta^z} that, for $B_1,B_2,\ldots,B_n \in \mathcal{B}(\mathbb{R})$, we have 
			
			\[
			\begin{array}{c}
			\mathbb{P}\left((\zeta^z_{t_1},\ldots,\zeta^z_{t_n})\in B_1\times \ldots \times B_n\vert \tau=r\right)=\mathbb{P}\left((X_{t_{1}}^{r,z},\ldots, X_{t_{n}}^{r,z})\in B_{1}\times\ldots\times B_{n}\right)= \\ \\\overset{n}{\underset{k=1}{\prod}}\delta_{z}(B_{k})\mathbb{I}_{\{r\leq t_{1}\}}+
			
			\sum\limits _{k=1}^{n-1}\mathbb{P}_{X_{t_{1}}^{r,z},\ldots,X_{t_{k}}^{r,z}}\left(B_{1}\times\ldots\times B_{k}\right)\overset{n}{\underset{j=k+1}{\prod}}\delta_{z}(B_{j})\mathbb{I}_{\{t_{k}<r\leq t_{k+1}\}}\\
			\\
			
			+\mathbb{P}_{X_{t_{1}}^{r,z},\ldots,X_{t_{n}}^{r,z}}\left(B_{1}\times\ldots\times B_{n}\right)\mathbb{I}_{\{t_{n}<r\}}
			\\ \\ =\dint_{B_{1}\times\ldots\times B_{n}}q_{t_{1},\ldots,t_{n}}(r,x_{1},\ldots,x_{n})\mu(dx_{1},\ldots,dx_{n}).
			
			\end{array}
			\]
			Here the function $q_{t_{1},\ldots,t_{n}}$ is a nonnegative and measurable in the $(n+1)$ variables jointly given by 
			\[
			\begin{array}{ll}
			q_{t_{1},\ldots,t_{n}}(r,x_{1},\ldots,x_{n})&= \overset{n}{\underset{k=1}{\prod}}\mathbb{I}_{\{z\}}(x_{k})\; \mathbb{I}_{\{r\leq t_{1}\}}\\
			&+\sum\limits _{k=1}^{n-1}\varphi_{X_{t_{1}}^{r,z},\ldots,X_{t_{k}}^{r,z}}\left(x_{1},\ldots,x_{k}\right)\; \overset{k}{\underset{j=1}{\prod}}\mathbb{I}_{\{x_j\neq z\}}\; \overset{n}{\underset{j=k+1}{\prod}}\mathbb{I}_{\{z\}}(x_{j})\; \mathbb{I}_{\{t_{k}<r\leq t_{k+1}\}}\\
			&+\varphi_{X_{t_{1}}^{r,z},\ldots,X_{t_{n}}^{r,z}}\left(x_{1},\ldots,x_{n}\right)\; \overset{n}{\underset{j=1}{\prod}}\mathbb{I}_{\{x_j \neq z\}}\mathbb{I}_{\{t_{n}<r\}}
			\end{array}
			\]
			and $\mu(dx_{1},dx_{2},\ldots,dx_{n})$ is a $\sigma$-finite measure on $\mathbb{R}^n$ given by
			\[
			\begin{array}{ll}
			\mu(dx_{1},dx_{2},\ldots,dx_{n})&=  \overset{n}{\underset{k=1}{\bigotimes}}\delta_{z}(dx_{k})\\
			 &+\sum\limits _{k=1}^{n-1}\lambda\left(dx_{1},\ldots,dx_{k}\right)\bigotimes\overset{n}{\underset{j=k+1}{\bigotimes}}\delta_{z}(B_{j})\\
			&+\lambda\left(dx_{1},\ldots,dx_{k}\right).
			\end{array}
			\]
			We get from Bayes formula  
			\[
			\mathbb{E}[g(\tau)\vert\zeta^z_{t_{1}},\ldots,\zeta^z_{t_{n}}]=\dfrac{\dint_{(0,+\infty )}g(r)q_{t_{1},\ldots,t_{n}}(r,\zeta^z_{t_{1}},\ldots,\zeta^z_{t_{n}})\mathbb{P}_{\tau}(dr)}{\dint_{(0,+\infty )}q_{t_{1},\ldots,t_{n}}(r,\zeta^z_{t_{1}},\ldots,\zeta^z_{t_{n}})\mathbb{P}_{\tau}(dr)}.
			\]
			By a simple integration we find
			\[
			\begin{array}{l}
			\dint_{(0,+\infty)}g(r)q_{t_{1},\ldots,t_{n}}(r,\zeta^z_{t_{1}},\ldots,\zeta^z_{t_{n}})\mathbb{P}_{\tau}(dr)  =\dint_{(0,t_{1}]}g(r)\mathbb{P}_{\tau}(dr)\mathbb{I}_{\{\zeta^z_{t_{1}}=z,\ldots,\zeta^z_{t_{n}}=z\}}\\
			\\
			+\sum\limits _{k=1}^{n-1}\dint_{(t_{k},t_{k+1}]}g(r)\varphi_{X_{t_{1}}^{r,z},\ldots,X_{t_{k}}^{r,z}}\left(\zeta^z_{t_{1}},\ldots,\zeta^z_{t_{k}}\right)\mathbb{P}_{\tau}(dr)\mathbb{I}_{\{\zeta^z_{t_{1}}\neq z,\ldots,\zeta^z_{t_{k}}\neq z,\zeta^z_{t_{k+1}}=z,\ldots,\zeta^z_{t_{n}}=z\}}\\
			\\
			+\dint_{(t_{n},+\infty)}g(r)\varphi_{X_{t_{1}}^{r,z},\ldots,X_{t_{n}}^{r,z}}\left(\zeta^z_{t_{1}},\ldots,\zeta^z_{t_{n}}\right)\mathbb{P}_{\tau}(dr)\mathbb{I}_{\{\zeta^z_{t_{1}}\neq z,\ldots,\zeta^z_{t_{n}}\neq z\}},
			\end{array}
			\]
			and 
			\[
			\begin{array}{l}
			\dint_{(0,+\infty)}q(r,\zeta^z_{t_{1}},\ldots,\zeta^z_{t_{n}})\mathbb{P}_{\tau}(dr)  =F(t_{1})\mathbb{I}_{\{\zeta^z_{t_{1}}=z,\ldots,\zeta^z_{t_{n}}=z\}}\\
			\\
			+\sum\limits _{k=1}^{n-1}\dint_{(t_{k},t_{k+1}]}\varphi_{X_{t_{1}}^{r,z},\ldots,X_{t_{k}}^{r,z}}\left(\zeta^z_{t_{1}},\ldots,\zeta^z_{t_{k}}\right)\mathbb{P}_{\tau}(dr)\mathbb{I}_{\{\zeta^z_{t_{1}}\neq z,\ldots,\zeta^z_{t_{k}}\neq z,\zeta^z_{t_{k+1}}=z,\ldots,\zeta^z_{t_{n}}=z\}}\\
			\\
			+\dint_{(t_{n},+\infty)}\varphi_{X_{t_{1}}^{r,z},\ldots,X_{t_{n}}^{r,z}}\left(\zeta^z_{t_{1}},\ldots,\zeta^z_{t_{n}}\right)\mathbb{P}_{\tau}(dr)\mathbb{I}_{\{\zeta^z_{t_{1}}\neq z,\ldots,\zeta^z_{t_{n}}\neq z\}}.
			\end{array}
			\]
		Using the fact that $\{\zeta^z_{t_{i}}=z\}\subset\{\zeta^z_{t_{j}}=z\}$ for $j\geq i$, we obtain
			\[
			\begin{array}{ll}
			\mathbb{E}[g(\tau)\vert\zeta^z_{t_{1}},\ldots,\zeta^z_{t_{n}}]&= \dint_{(0,t_{1}]}\frac{g(r)}{F(t_{1})}\mathbb{P}_{\tau}(dr)\mathbb{I}_{\{\zeta^z_{t_{1}}=z\}}\\
			 &+\sum\limits _{k=1}^{n-1}\dint_{(t_{k},t_{k+1}]}g(r)\dfrac{\varphi_{X_{t_{1}}^{r,z},\ldots,X_{t_{k}}^{r,z}}\left(\zeta^z_{t_{1}},\ldots,\zeta^z_{t_{k}}\right)}{\dint_{(t_{k},t_{k+1}]}\varphi_{X_{t_{1}}^{s,z},\ldots,X_{t_{k}}^{s,z}}\left(\zeta_{t_{1}}^z,\ldots,\zeta^z_{t_{k}}\right)\mathbb{P}_{\tau}(ds)}\mathbb{P}_{\tau}(dr)\mathbb{I}_{\{\zeta^z_{t_{k}}\neq z,\zeta^z_{t_{k+1}}=z\}}\\
		 &+\dint_{(t_{n},+\infty)}g(r)\dfrac{\varphi_{X_{t_{1}}^{r,z},\ldots,X_{t_{n}}^{r,z}}\left(\zeta^z_{t_{1}},\ldots,\zeta^z_{t_{n}}\right)}{\dint_{(t_{n},+\infty)}\varphi_{X_{t_{1}}^{s,z},\ldots,X_{t_{n}}^{s,z}}\left(\zeta^z_{t_{1}},\ldots,\zeta^z_{t_{n}}\right)\mathbb{P}_{\tau}(ds)}\mathbb{P}_{\tau}(dr)\mathbb{I}_{\{\zeta^z_{t_{n}}\neq z\}}.
			\end{array}
			\]
			Now for $k=1,...,n-1$ we use \eqref{equationfinidimensionaldensity} to see that
			
			\begin{align*}
			\dfrac{\varphi_{X_{t_{1}}^{r,z},\ldots,X_{t_{k}}^{r,z}}\left(x_{1},\ldots,x_{k}\right)}{\dint_{(t_{k},t_{k+1}]}\varphi_{X_{t_{1}}^{s,z},\ldots,X_{t_{k}}^{s,z}}\left(x_{1},\ldots,x_{k}\right)\mathbb{P}_{\tau}(ds)}&=\dfrac{\dfrac{f_{r-t_{k}}(z-x_{k})}{f_r(z)}}{\dint_{(t_{k},t_{k+1}]}\dfrac{f_{s-t_{k}}(z-x_{k})}{f_s(z)}\mathbb{P}_{\tau}(ds)}\\
			\\
			&=\Upsilon_{k}(r,x_{k}),
			\end{align*}
			
			and
\begin{align*}
			\dfrac{\varphi_{X_{t_{1}}^{r,z},\ldots,X_{t_{n}}^{r,z}}\left(x_{1},\ldots,x_{n}\right)}{\dint_{(t_{n},+\infty)}\varphi_{X_{t_{1}}^{s,z},\ldots,X_{t_{n}}^{s,z}}\left(x_{1},\ldots,x_{n}\right)\mathbb{P}_{\tau}(ds)}
			&=\dfrac{\dfrac{f_{r-t_{n}}(z-x_{n})}{f_r(z)}}{\dint_{(t_,+\infty)}\dfrac{f_{s-t_{n}}(z-x_{n})}{f_s(z)}\mathbb{P}_{\tau}(ds)}\\
			\\
			&=\Upsilon_{n}(r,x_{n}).
			\end{align*}	
			Combining all this leads to the formula \eqref{taucondxi}.
		\end{proof}
			\begin{corollary} The conditional law of the random time $\tau$ given $\zeta^z_{t_{1}},\ldots,\zeta^z_{t_{n}}$ is given by 
					\begin{align}\label{taucondlawmulti}
					\mathbb{P}_{\tau\vert\zeta^z_{t_{1}}=x_{1},\ldots,\zeta^z_{t_{n}}=x_{n}}\left(x_{1},\ldots,x_{n},dr\right)= & \dfrac{1}{F(t_{1})}\,\mathbb{I}_{\{x_{1}=z\}}\,\mathbb{I}_{(0,t_{1}]}(r)\,\mathbb{P}_{\tau}(dr) \nonumber\\
					\nonumber \\
					 +&\sum\limits _{k=1}^{n-1}\Upsilon_{k}(r,x_{k})\,\mathbb{I}_{\{x_{k}\neq z,x_{k+1}=z\}}\,\mathbb{I}_{(t_{k},t_{k+1}]}(r)\,\mathbb{P}_{\tau}(dr) \nonumber\\
					\nonumber \\
					+&\Upsilon_{n}(r,x_{n})\,\mathbb{I}_{\{x_{n}\neq z\}}\,\mathbb{I}_{(t_{n},+\infty)}(r)\,\mathbb{P}_{\tau}(dr).
					\end{align}
			\end{corollary}
			\subsection{Markov property with respect the natural filtration}
	Now the main objective is to prove that $\zeta^z$ is a Markov process with respect its natural filtration $\mathbb{F}^{\zeta^z}$. In order to reach this goal, we first need the expression of the conditional law of $\zeta^z_t$, for $t>0$, with respect to the past $n$ values of $\zeta^z$ before time $t$. On a more general note, we give the conditional law of $(\tau,\zeta^z_.)$ with respect to past $n$ values of $\zeta^z$ before time $t$ which is the content of

	\begin{proposition}\label{ncordbayesest}
      Let $n\in \mathbb{N}^*$ and $0=t_0<t_1<t_2<...<t_n<u$ such that $F(t_1)>0$.  Let $\mathfrak{g}$ be a bounded measurable function defined on $(0,+\infty)\times \mathbb{R}$. Then, $\mathbb{P}$-a.s., we have
		\begin{align}
			\mathbb{E}[\mathfrak{g}(\tau,\zeta^z_u)|\zeta^z_{t_1},\ldots,\zeta^z_{t_n}]=& \dint_{(0,t_1]}\frac{\mathfrak{g}(r,z)}{F(t_1)}\mathbb{P}_\tau(dr)\mathbb{I}_{\{\zeta_{t_{1}}=z\}} \nonumber \\
			\nonumber \\	 +&\sum\limits _{k=1}^{n-1}\dint_{(t_{k},t_{k+1}]}\mathfrak{g}(r,z)\Upsilon_{k}(r,\zeta^z_{t_{k}})\mathbb{P}_{\tau}(dr)\mathbb{I}_{\{\zeta^z_{t_{k}}\neq z,\zeta^z_{t_{k+1}}=z\}} \nonumber  \\
			\nonumber \\		 	+& \dint_{(t_{n},u]}\mathfrak{g}(r,z)\Upsilon_{n}(r,\zeta^z_{t_{n}})\mathbb{P}_{\tau}(dr)\mathbb{I}_{\{\zeta^z_{t_{n}}\neq z\}} \nonumber  \\
			\nonumber \\	 +&\dint_{(u,+\infty)}\,\mathfrak{G}_{t_{n},u}(r,\zeta^z_{t_{n}})\Upsilon_{n}(r,\zeta^z_{t_{n}})\,\mathbb{P}_{\tau}(dr)\,\mathbb{I}_{\{\zeta^z_{t_{n}}\neq z\}} .	\label{equationbeyesextensionxi_u,t_n<u}
			\end{align}	
		  Here, for all $0<t<u<r$, the function $\mathfrak{G}_{t,u}$ is defined by
		  \begin{eqnarray}
		  \mathfrak{G}_{t,u}(r,x)&:=& \mathbb{E}[\mathfrak{g}(r,X_{u}^{r,z})|X_{t}^{r,z}=x] \nonumber\\
		  &=& \dint_{\mathbb{R}} \mathfrak{g}(r,y)\dfrac{f_{u-t}(y-x)f_{r-u}(z-y)}{f_{r-t}(z-x)}dy. \label{Gturx}
		  \end{eqnarray}
		  In particular when $\mathfrak{g}$ is defined on $\mathbb{R}$ we have the following 
		
			\begin{align}
			\mathbb{E}[g(\zeta^z_u)|\zeta^z_{t_1},\ldots,\zeta^z_{t_n}]=& \,g(z) \left( \mathbb{I}_{\{\zeta^z_{t_{n}}=z\}} 	+ \dint_{(t_{n},u]}\,\Upsilon_{n}(r,\zeta^z_{t_{n}})\mathbb{P}_{\tau}(dr)\mathbb{I}_{\{\zeta^z_{t_{n}}\neq z\}} \right)\nonumber  \\
			\nonumber \\ +&\dint_{(u,+\infty)}\,K_{t_{n},u}(r,\zeta^z_{t_{n}})\Upsilon_{n}(r,\zeta^z_{t_{n}})\,\mathbb{P}_{\tau}(dr)\,\mathbb{I}_{\{\zeta^z_{t_{n}}\neq z\}}	, \label{equationzeta_u/zeta_tn}
			\end{align}	
		where,  for all $t<u<r$, the function $K_{t,u}$ is given by
		\begin{eqnarray}
		K_{t,u}(r,x)&:=& \mathbb{E}[g(X_{u}^{r,z})|X_{t}^{r,z}=x] \nonumber\\
		&=& \dint_{\mathbb{R}} g(y)\dfrac{f_{u-t}(y-x)f_{r-u}(z-y)}{f_{r-t}(z-x)}dy. \label{Hcondlaw}
		\end{eqnarray}

	\end{proposition}

\begin{proof}
	\begin{itemize}
		 
  First we split $\mathbb{E}[\mathfrak{g}(\tau,\zeta^z_{u})\vert\zeta^z_{t_{1}},\ldots,\zeta^z_{t_{n}}]$ as follows
	\[
\begin{array}{ll}
\mathbb{E}[\mathfrak{g}(\tau,\zeta^z_{u})\vert\zeta^z_{t_{1}},\ldots,\zeta^z_{t_{n}}]&= \mathbb{E}[\mathfrak{g}(\tau,z)\mathbb{I}_{\{\tau\leq t_{n}\}}\vert\zeta^z_{t_{1}},\ldots,\zeta^z_{t_{n}}]\\
\\
 &+ \mathbb{E}[\mathfrak{g}(\tau,z)\mathbb{I}_{\{t_{n}<\tau\leq u\}}|\zeta^z_{t_{1}},\ldots,\zeta^z_{t_{n}}]\\
\\
 & +\mathbb{E}[\mathfrak{g}(\tau,\zeta^z_{u})\mathbb{I}_{\{u<\tau\}}|\zeta^z_{t_{1}},\ldots,\zeta^z_{t_{n}}].
\end{array}
\]
We obtain from Proposition \eqref{propbayesestimatejusquatn} that
\[
\begin{array}{ll}
\mathbb{E}[\mathfrak{g}(\tau,z)\mathbb{I}_{\{\tau\leq t_{n}\}}\vert\zeta^z_{t_{1}},\ldots,\zeta^z_{t_{n}}]&= \dint_{(0,t_{1}]}\frac{\mathfrak{g}(r,z)}{F(t_{1})}\mathbb{P}_{\tau}(dr)\mathbb{I}_{\{\zeta^z_{t_{1}}=z\}}\\
\\
 & +\sum\limits _{k=1}^{n-1}\dint_{(t_{k},t_{k+1}]}\mathfrak{g}(r,z)\Upsilon_{k}(r,\zeta^z_{t_{k}})\mathbb{P}_{\tau}(dr)\mathbb{I}_{\{\zeta^z_{t_{k}}\neq  z,\zeta^z_{t_{k+1}}=z\}},
\end{array}
\]
and
\[
\begin{array}{ll}
\mathbb{E}[\mathfrak{g}(\tau,z)\mathbb{I}_{\{t_{n}<\tau\leq u\}}\vert\zeta^z_{t_{1}},\ldots,\zeta^z_{t_{n}}]= & \dint_{(t_{n},u]}\mathfrak{g}(r,z)\Upsilon_{n}(r,\zeta^z_{t_{n}})\mathbb{P}_{\tau}(dr)\mathbb{I}_{\{\zeta^z_{t_{n}}\neq z\}}.\end{array}
\]
Next we show that 
\begin{equation}
\mathbb{E}[\mathfrak{g}(\tau,\zeta^z_{u})\mathbb{I}_{\{u<\tau\}}|\zeta^z_{t_{1}},\ldots,\zeta^z_{t_{n}}]=\dint_{(u,+\infty)}\,\mathfrak{G}_{t_{n},u}(r,\zeta^z_{t_{n}})\Upsilon_{n}(r,\zeta^z_{t_{n}})\,\mathbb{P}_{\tau}(dr)\,\mathbb{I}_{\{\zeta^z_{t_{n}}\neq z\}}.	\label{condeqtauxi}
\end{equation}
Indeed, since $X^{r,z}$  is a Markov process then for a bounded Borel function $h$ we have  
\begin{align*}
\mathbb{E}[\mathfrak{g}(\tau,\zeta^z_{u})\mathbb{I}_{\left\{ u<\tau\right\} }h\left(\zeta^z_{t_{1}},\ldots,\zeta^z_{t_{n}}\right)] & =\dint_{(u,+\infty)}E[\mathfrak{g}(r,X_{u}^{r,z})h(X_{t_{1}}^{r,z},\ldots,X_{t_{n}}^{r,z})]\mathbb{P}_{\tau}(dr)\\
 & =\dint_{(u,+\infty)}\mathbb{E}[\mathbb{E}[\mathfrak{g}(r,X_{u}^{r,z})|X_{t_{n}}^{r,z}]h(X_{t_{1}}^{r,z},\ldots,X_{t_{n}}^{r,z})]\mathbb{P}_{\tau}(dr).
\end{align*}
 Using \eqref{Gturx}, for $t_{n}<u<r$, we get  
\[
\begin{array}{lll}
\mathbb{E}[\mathfrak{g}(\tau,\zeta^z_{u})\mathbb{I}_{\{u<\tau\}}h(\zeta^z_{t_{1}},\ldots,\zeta^z_{t_{n}})] & = & \dint_{(u,+\infty)}\mathbb{E}[\mathfrak{G}_{t_{n},u}(r,X_{t_{n}}^{r,z})h(X_{t_{1}}^{r,z},\ldots,X_{t_{n}}^{r,z})]\mathbb{P}_{\tau}(dr)\\
\\
 & = & \mathbb{E}[\mathfrak{G}_{t_{n},u}(\tau,\zeta^z_{t_{n}})\mathbb{I}_{\{u<\tau\}}h(\zeta^z_{t_{1}},\ldots,\zeta^z_{t_{n}})].
\end{array}
\]
It follows from Proposition \eqref{propbayesestimatejusquatn} that for every $(x_1,\ldots,x_n)\in \mathbb{R}^n$, we have

\begin{align*}
\mathbb{E}[\mathfrak{G}_{t_{n},u}(\tau,\zeta^z_{t_{n}})\mathbb{I}_{\{u<\tau\}}\vert\zeta^z_{t_{1}}=x_1,\ldots,\zeta^z_{t_{n}}=x_n]&=\mathbb{E}[\mathfrak{G}_{t_{n},u}(\tau,x_n)\mathbb{I}_{\{u<\tau\}}\vert\zeta^z_{t_{1}}=x_1,\ldots,\zeta^z_{t_{n}}=x_n]\nonumber\\ &=\dint_{(u,+\infty)}\,\mathfrak{G}_{t_{n},u}(r,x_n)\,\Upsilon_{n}(r,\zeta^z_{t_{n}})\,\mathbb{P}_{\tau}(dr)\,\mathbb{I}_{\{x_n\neq z\}}.
\end{align*}
This induces that
\[
\begin{array}{l}
\mathbb{E}[\mathfrak{g}(\tau,\zeta^z_{u})\mathbb{I}_{\{u<\tau\}}h(\zeta^z_{t_{1}},\ldots,\zeta^z_{t_{n}})]=\\
\\
\mathbb{E}\left(\dint_{(u,+\infty)}\,\mathfrak{G}_{t_{n},u}(r,\zeta^z_{t_{n}})\Upsilon_{n}(r,\zeta^z_{t_{n}})\,\mathbb{P}_{\tau}(dr)\,\mathbb{I}_{\{\zeta^z_{t_{n}}\neq z\}}h(\zeta^z_{t_{1}},\ldots,\zeta^z_{t_{n}})\right).
\end{array}
\]
Hence the formula \eqref{condeqtauxi} is proved and then the proof of the proposition is completed.
\end{itemize}
	To obtain \eqref{equationzeta_u/zeta_tn}, we use \eqref{equationbeyesextensionxi_u,t_n<u} with the facts that $\{\zeta^z_{t_{i}}=z\}\subset\{\zeta^z_{t_{j}}=z\}$ for $j\geq i$ and $$\dint_{(t_{k},t_{k+1}]}\Upsilon_{k}(r,\zeta^z_{t_{k}})\mathbb{P}_{\tau}(dr)=1$$ for $k=1,\ldots,n-1$.
\end{proof}	
\begin{remark}	
We can derive easily two main facts from \eqref{equationzeta_u/zeta_tn}. The first one is
\begin{align}
			\mathbb{E}[g(\zeta^z_u)|\zeta^z_{t_n}]=& \,g(z) \left( \mathbb{I}_{\{\zeta^z_{t_{n}}=z\}} 	+ \dint_{(t_{n},u]}\,\Upsilon_{n}(r,\zeta^z_{t_{n}})\mathbb{P}_{\tau}(dr)\mathbb{I}_{\{\zeta^z_{t_{n}}\neq z\}} \right)\nonumber  \\
			\nonumber \\ +&\dint_{(u,+\infty)}\,K_{t_{n},u}(r,\zeta^z_{t_{n}})\Upsilon_{n}(r,\zeta^z_{t_{n}})\,\mathbb{P}_{\tau}(dr)\,\mathbb{I}_{\{\zeta^z_{t_{n}}\neq z\}}	.\label{cond u/tn}
			\end{align} 
The other thing is 
\begin{equation}\label{condu/t1-tn}
\mathbb{E}[g(\zeta^z_u)|\zeta^z_{t_1},\ldots,\zeta^z_{t_n}]=\mathbb{E}[g(\zeta^z_u)|\zeta^z_{t_n}].
\end{equation}
	
\end{remark}

		Our next task to establish the markov property of $\zeta$ with respect to its natural filtration. 
			
			\begin{theorem}\label{thmxitaumarkov}
				The L\'evy bridge $\zeta^z$ of random length $\tau$  is an $\mathbb{F}^{\zeta^z}$-Markov process with transition law given by
						\begin{align}\label{transprob}
			\mathbb{P}(\zeta_u^z\in dy|\zeta_t^z=x)&=\Bigg[  \mathbb{I}_{\{x=z\}}\mathbb{I}_{\{y=z\}}+\dint_{(t,u]}\,\dfrac{\phi(r,t,x)}{\dint_{(t,+\infty)}\phi(r,t,x)\mathbb{P}_{\tau}(dr)}\,\mathbb{P}_{\tau}(dr)\,\mathbb{I}_{\{x\neq z\}}\mathbb{I}_{\{y=z\}}
 \nonumber \\
		&+f_{u-t}(y-x)\,\dint_{(u,+\infty)}\,\dfrac{\phi(r,u,y)}{\dint_{(t,+\infty)}\phi(r,t,x)\mathbb{P}_{\tau}(dr)}\,\mathbb{P}_{\tau}(dr)\,\mathbb{I}_{\{x\neq z\}}\mathbb{I}_{\{y\neq z\}}
\Bigg]  \varrho(dy),
			\end{align}
				for $0 \leq t <u$, where $\varrho$ is the measure defined on $\mathcal{B}(\mathbb{R})$ by
				\begin{equation}\label{measure}
				 \varrho(dy)=\delta_z(dy)+\lambda (dy).
				\end{equation} 
			\end{theorem} 	
	\begin{proof}
		We need to prove that for all $0\leq t <u$, and all measurable functions $g$ such that $g(\zeta_u^z)$ is integrable, we have, $\mathbb{P}$-a.s,
		 $$\mathbb{E}[g(\zeta^z_u)\vert \mathcal{F}^{\zeta^z}_{t} ]=\mathbb{E}[g(\zeta^z_u)\vert \zeta^z_{t}]. $$
		First, we would like to mention that since $\zeta^z_{0}=0$ almost surely then it is easy to see that 
		$$\mathbb{E}[g(\zeta^z_u)\vert \mathcal{F}^{\zeta^z}_{0} ]=\mathbb{E}[g(\zeta^z_u)\vert \zeta^z_{0}].$$	
		Using Theorem 1.3 in Blumenthal and Getoor \cite{BG} it suffices to prove that for all $n\in \mathbb{N}$ and
		$0 < t_1 < t_2 <\ldots < t_n < u$, we have $\mathbb{P}$-a.s.,
		\begin{equation}
			\mathbb{E}[g(\zeta^z_u)|\zeta^z_{t_1},\ldots,\zeta^z_{t_n}]=\mathbb{E}[g(\zeta^z_u)|\zeta^z_{t_n}]
		\end{equation}
		which is none other than \eqref{condu/t1-tn}.
Now it is sufficient to write down \eqref{equationphi} and \eqref{cond u/tn} with $t$ instead of $t_n$ to conclude that, for all $0<t<u$, 

\begin{align}\label{condu/t}
\mathbb{E}[g(\zeta_{u}^{z})|\zeta_{t}^{z}] = &\,  g(z)\left(\mathbb{I}_{\{\zeta_{t}^{z}=z\}}+\dint_{(t,u]}\,\dfrac{\phi(r,t,\zeta_{t}^{z})}{\dint_{(t,+\infty)}\phi(r,t,\zeta_{t}^{z})\mathbb{P}_{\tau}(dr)}\mathbb{P}_{\tau}(dr)\mathbb{I}_{\{\zeta_{t}^{z}\neq z\}}\right) \nonumber\\
\nonumber\\
 &   +\dint_{\mathbb{R}}g(y)\,f_{u-t}(y-\zeta_{t}^{z})\,\int_{(u,+\infty)}\,\dfrac{\phi(r,u,y)}{\dint_{(t,+\infty)}\phi(r,t,\zeta_{t}^{z})\mathbb{P}_{\tau}(dr)}\,\mathbb{P}_{\tau}(dr)\,dy\,\mathbb{I}_{\{\zeta_{t}^{z}\neq z\}}.
\end{align}
Thereafter the conditional law of $\zeta^z_u$ given $\zeta^z_t$ is given by
			\begin{align*}
			\mathbb{P}(\zeta_u^z\in dy|\zeta_t^z=x)&=\Bigg[  \mathbb{I}_{\{x=z\}}\mathbb{I}_{\{y=z\}}+\dint_{(t,u]}\,\dfrac{\phi(r,t,x)}{\dint_{(t,+\infty)}\phi(r,t,x)\mathbb{P}_{\tau}(dr)}\,\mathbb{P}_{\tau}(dr)\,\mathbb{I}_{\{x\neq z\}}\mathbb{I}_{\{y=z\}}
\\
		&+f_{u-t}(y-x)\,\dint_{(u,+\infty)}\,\dfrac{\phi(r,u,y)}{\dint_{(t,+\infty)}\phi(r,t,x)\mathbb{P}_{\tau}(dr)}\,\mathbb{P}_{\tau}(dr)\,\mathbb{I}_{\{x\neq z\}}\mathbb{I}_{\{y\neq z\}}
\Bigg]  \varrho(dy),
			\end{align*}
		where $\varrho$ is the measure given by \eqref{measure}.		
	\end{proof}
	
	\begin{remark}
		\begin{enumerate}
			\item[(i)] 	It is not hard to see that the Markov property can be extended to the completed filtration  $\mathbb{F}^{\zeta^z,c}$.
			\item[(ii)] It is clear that the transition density $ \mathbb{P}(\zeta_u^z\in dy|\zeta_t^z=x)$ doesn't depend only on $u - t$,	then the process
			$\zeta^z$ cannot be an homogeneous $\mathbb{F}^{\zeta^z}$-Markov process. 
		\end{enumerate}	
	\end{remark}
In the next proposition we give the conditional law of the random time $\tau$ with respect to the
past of the process $\zeta^z$ up to time $t$.
\begin{proposition}\label{propbayesestimate}
	Let $t>0$. For each Borel function difined on $(0, \infty)$ such that $\mathbb{E}[|g(\tau)|]<+\infty$ we have $\mathbb{P}$-a.s.
	\begin{equation}
	\mathbb{E}[g(\tau)\vert  \mathcal{F}_{t}^{\zeta^z,c}]=g(\tau\wedge t)\mathbb{I}_{\{\zeta^z_t=z\}}+\dint_{(t,+\infty)}g(r)\,\dfrac{\phi(r,t,\zeta_{t}^{z})}{\dint_{(t,+\infty)}\phi(r,t,\zeta_{t}^{z})\mathbb{P}_{\tau}(dr)}\,\mathbb{P}_{\tau}(dr)\;\mathbb{I}_{\{\zeta_{t}^{z}\neq z\}}.\label{equationtaugivenF}
	\end{equation}
\end{proposition}
\begin{proof}
		Obviously, we have
		$$\mathbb{E}[g(\tau)| \mathcal{F}_{t}^{\zeta^z,c}]=\mathbb{E}[g(\tau\wedge t)\mathbb{I}_{\{\tau\leqslant t\}}|\mathcal{F}_{t}^{\zeta^z,c}]+\mathbb{E}[g(\tau\vee t)\mathbb{I}_{\{ t<\tau \}}|\mathcal{F}_{t}^{\zeta^z,c}].$$
		Now since  $g(\tau\wedge t)\mathbb{I}_{\{\tau\leqslant t\}}$ is $ \mathcal{F}^{\zeta^z,c}_{t}$-measurable then, $\mathbb{P}$-a.s, one has
		\[
		\begin{array}{lll}
		\mathbb{E}[g(\tau\wedge t)\mathbb{I}_{\{\tau\leqslant t\}}\vert\mathcal{F}_{t}^{\zeta^z,c}] & = & g(\tau\wedge t)\mathbb{I}_{\{\tau\leqslant t\}}\\
		\\
		& = & g(\tau\wedge t)\mathbb{I}_{\{\zeta^z_{t}=z\}}.
		\end{array}
		\]
		On the other hand it follows from the facts that $\zeta^z$ is a Markov process with respect to its completed natural filtration and $g(\tau\vee t)\mathbb{I}_{\{t< \tau \}}$ is  $\sigma(\zeta^z_s, t \leq s \leq +\infty) \vee \mathcal{N}_P$-measurable that $\mathbb{P}$-a.s.
		$$\mathbb{E}[g(\tau\vee t)\mathbb{I}_{\{ t<\tau \}}\vert \mathcal{F}_{t}^{\zeta^z,c}]=\mathbb{E}[g(\tau\vee t)\mathbb{I}_{\{ t<\tau \}}|\zeta^z_{t}].$$
		Thus the result is deduced from \eqref{taucondxi} with $t$ instead of $t_n$.
\end{proof}
 \begin{remark}
 	Analogously, using the formula \eqref{equationbeyesextensionxi_u,t_n<u} and the Markov property of $\zeta^z$, we obtain $\mathbb{P}$-a.s.
 		 \begin{multline}
 		\mathbb{E}[g(\tau,\zeta^z_{u})|\mathcal{F}_{t}^{\zeta^z,c}] =g(\tau\wedge t,z)\mathbb{I}_{\{\zeta^z_t=z\}}+\dint_{(t,u]}\,g(r,z)\,\dfrac{\phi(r,t,\zeta_{t}^{z})}{\dint_{(t,+\infty)}\phi(r,t,\zeta_{t}^{z})\mathbb{P}_{\tau}(dr)}\,\mathbb{P}_{\tau}(dr)\;\mathbb{I}_{\{\zeta_{t}^{z}\neq z\}}
\\
 		+\dint_{\mathbb{R}}\,\dint_{(u,+\infty)}\,g(r,y)\,f_{u-t}(y-\zeta_{t}^{z})\,\,\dfrac{\phi(r,u,y)}{\dint_{(t,+\infty)}\phi(r,t,\zeta_{t}^{z})\mathbb{P}_{\tau}(dr)}\,\mathbb{P}_{\tau}(dr)\,dy\,\mathbb{I}_{\{\zeta_{t}^{z}\neq z\}}
, \label{equationtauzetaugivenF}
 		\end{multline}
 for each bounded measurable function $g$ defined on $(0,+\infty)\times \mathbb{R}$ and $0<t<u$.
 \end{remark}
\subsection{The right continuity of the filtration $\mathbb{F}^{\zeta^z,c}_+$}	\label{sectionmarkovpropertyfiltrationcontinue}
We have established, in the previous section, the Markov property of $\zeta^z$ with respect to its completed natural filtration $\mathbb{F}^{\zeta^z,c}$. In this section we are interested in the right-continuity of its completed natural filtration $\mathbb{F}^{\zeta^z,c}$. 
We will need to impose the following condition on the transition density
$f_{t}$ of $X_{t}$:

\begin{hy}\label{behavioursmall}
For any $\delta>0$ there exists some $\hat{t}:=\hat{t}\left(\delta\right)>0$ such that $$C_{\delta}=\underset{0<t<\hat{t}}{\sup}\,\underset{\vert x\vert>\delta}{\sup}\,\vert f_{t}(x)\vert<+\infty.$$ 
\end{hy}
We will shed some light on this assumption. Clearly this condition is closely related to the behavior of the transition density of the Lévy process in a short time. Estimates of the transition density of Lévy processes have been the subject of a number of investigations. Léandre \cite{Lea87} and Ishikawa \cite{Ish} were the first to show that the transition density behaves like a power of $t$ when $t$ goes to $0$, which leads to this assumption. To learn more about this condition we refer to the paper of Figueroa-L$\acute{\text{o}}$pez et al \cite{FH}.
 
In order to be able to state the main result we need the following auxiliary result:
\begin{lemma}\label{lemmatransitionfunctionrigthcontinuous}	Let $u$ be a strictly positive real number and $g$ be a bounded continuous function. Assume that assumption \eqref{behavioursmall} holds. Then, $\mathbb{P}$-a.s., we have
	\begin{enumerate}
		\item[(i)] The function $t\longmapsto \mathbb{E}[g(\zeta^z_u)|\zeta^z_t]$ is right-continuous on $]0,u]$.
		\item[(ii)] If, in addition, $\mathbb{P}(\tau>\varepsilon)=1$ for some $\varepsilon>0$. Then the function $t\longmapsto \mathbb{E}[g(\zeta^z_u)|\zeta^z_t]$ is right-continuous at $0$.
	\end{enumerate}
\end{lemma}
\begin{proof}
(i)	Let $t\in]0,u]$ and $(t_{n})_{n\in \mathbb{N}}$
	be a decreasing sequence of real numbers converging to $t$. Our goal is to show that 
	\begin{equation}
	\lim\limits_{n\longmapsto+\infty}\mathbb{E}[g(\zeta^z_{u})\vert\zeta^z_{t_n}]=\mathbb{E}[g(\zeta^z_{u})\vert\zeta^z_{t}].\label{eqmarkovlimitzeta}
	\end{equation}
	It follows from \eqref{stpprop} and \eqref{cond u/tn} that $\mathbb{P}$-a.s.
	
	\begin{align}
	\mathbb{E}[g(\zeta^z_u)|\zeta^z_{t_n}]=& \,g(z) \left( \mathbb{I}_{\{\tau \leq t_{n}\} }	 	+ \dint_{(t_{n},u]}\,\Upsilon_{n}(r,\zeta^z_{t_{n}})\mathbb{P}_{\tau}(dr)\mathbb{I}_{\{ t_{n}<\tau\} } \right)\nonumber  \\
			\nonumber \\ +&\dint_{(u,+\infty)}\,K_{t_{n},u}(r,\zeta^z_{t_{n}})\Upsilon_{n}(r,\zeta^z_{t_{n}})\,\mathbb{P}_{\tau}(dr)\,\mathbb{I}_{\{ t_{n}<\tau\} }.
	\end{align}
It is clear from \eqref{condu/t} that \eqref{eqmarkovlimitzeta} holds true if the
		following two identities are satisfied $\mathbb{P}$-a.s. on $\{t<\tau\}$:
	\begin{equation}
	\lim\limits_{n \longrightarrow +\infty}\dint_{(t_{n},u]}\,\Upsilon_{n}(r,\zeta^z_{t_{n}})\mathbb{P}_{\tau}(dr)=\dint_{(t,u]}\,\dfrac{\phi(r,t,\zeta_{t}^{z})}{\dint_{(t,+\infty)}\phi(r,t,\zeta_{t}^{z})\mathbb{P}_{\tau}(dr)}\mathbb{P}_{\tau}(dr),\label{equationlimitphitnxi-phitxi}
	\end{equation}
	and

\begin{eqnarray}\label{secondterm}
\lim\limits _{n\longrightarrow+\infty}\dint_{(u,+\infty)}\,K_{t_{n},u}(r,\zeta_{t_{n}}^{z})\Upsilon_{n}(r,\zeta_{t_{n}}^{z})\,\mathbb{P}_{\tau}(dr)= \nonumber\\
\nonumber\\
\dint_{\mathbb{R}}g(y)\,f_{u-t}(y-\zeta_{t}^{z})\,\dint_{(u,+\infty)}\,\dfrac{\phi(r,u,y)}{\dint_{(t,+\infty)}\phi(r,t,\zeta_{t}^{z})\mathbb{P}_{\tau}(dr)}\,&\mathbb{P}_{\tau}(dr)\,dy.
\end{eqnarray}

%
		We start by proving	 assertion \eqref{equationlimitphitnxi-phitxi}. From \eqref{equationphi} the left-hand side of \eqref{equationlimitphitnxi-phitxi} can be rewritten as

\[
\dint_{(t_{n},u]}\Upsilon_{n}(r,\zeta^z_{t_{n}})\mathbb{P}_{\tau}(dr)=\frac{\dint_{(t_{n},u]}\phi(r,t_{n},\zeta^z_{t_{n}})\,\mathbb{P}_{\tau}(dr)}{\dint_{(t_{n},+\infty)}\phi(r,t_{n},\zeta^z_{t_{n}})\,\mathbb{P}_{\tau}(ds)}=\dfrac{\dint_{(t_{n},u]}\dfrac{f_{r-t_{n}}(z-\zeta^z_{t_{n}})}{f_{r}(z)}\,\mathbb{P}_{\tau}(dr)}{\dint_{(t_{n},+\infty)}\dfrac{f_{r-t_{n}}(z-\zeta^z_{t_{n}})}{f_{r}(z)}\,\mathbb{P}_{\tau}(dr)}.
\]		
		
		Keeping in mind that hypothesis \eqref{hyintegr} leads to the joint continuity of the function $(t,x)\longrightarrow f_t(x)$ on $(0, \infty)\times \mathbb{R}$, we conclude that the function 
		\[
		(t,r,x)\longrightarrow\dfrac{f_{r-t}(z-x)}{f_r(z)}\mathbb{I}_{\left\{ t<r\right\} }
		\]
		defined on $(0,+\infty)\times[0,+\infty)\times\mathbb{R}\backslash\{z\}$
		is continuous. As the paths of $\zeta^z$ are càdlàg, we further have $\mathbb{P}$-a.s on $\left\{ t<\tau\right\} $
		\begin{equation}
		\underset{n\rightarrow+\infty}{\lim}\,\dfrac{f_{r-t_{n}}(z-\zeta^z_{t_{n}})}{f_{r}(z)}\mathbb{I}_{\left\{ t_{n}<r\right\} }=\dfrac{f_{r-t}(z-\zeta^z_{t})}{f_{r}(z)}\mathbb{I}_{\left\{ t<r\right\} }.\label{phixiconv}
		\end{equation}
		On the other hand assertion (ii) of Corollary \ref{corschilling} implies that
		\begin{equation}\label{convinfinie}
		\lim\limits_{r\longrightarrow \infty}\dfrac{f_{r-t}(z-x)}{f_r(z)}=1,
		\end{equation}
	locally uniformly in $x$. Then from assumption \eqref{behavioursmall}, the joint continuity of $f_t(x)$ and \eqref{convinfinie} we get
		\[
		\underset{(t,x)\in\mathcal{K},r>0}{\sup}\,\dfrac{f_{r-t}(z-x)}{f_r(z)}\mathbb{I}_{\left\{ t<r\right\} }<+\infty,
		\]
for any compact subset $\mathcal{K}$ of $(0,+\infty)\times\mathbb{R}\backslash\{z\}$. It results, $\mathbb{P}$-a.s on $\left\{ t<\tau\right\} $, that
		\begin{equation}
		\underset{n\in\mathbb{N},r>t}{\sup}\,\dfrac{f_{r-t_n}(z-\zeta^z_{t_n})}{f_r(z)}\mathbb{I}_{\left\{ t_n<r\right\} }<+\infty.\label{varphiconvergence}
		\end{equation}
		We conclude assertion \eqref{equationlimitphitnxi-phitxi} from the Lebesgue dominated convergence theorem.\\
		Now let us prove \eqref{secondterm}. Let us first recall that 
		$$K_{t_{n},u}(r,\zeta^z_{t_{n}})=\dint_{\mathbb{R}} g(y)\dfrac{f_{u-t_n}(y-\zeta^z_{t_n})f_{r-u}(z-y)}{f_{r-t_n}(z-\zeta^z_{t_n})}dy.$$
		Since $g$ is bounded we deduce that $K_{t_{n},u}(r,\zeta_{t_{n}})$ is bounded. Moreover, as in Remark \ref{ratio}, the function $$ y\longrightarrow \dfrac{f_{u-t_{n}}(y-\zeta^z_{t_n})f_{r-u}(z-y)}{f_{r-t_{n}}(z-\zeta^z_{t_n})} $$
	is a well defined probability density function for almost every value of $\zeta^z_{t_n}$. This provides a sequence of density functions that converge, $\mathbb{P}$-a.s on $\left\{ t<\tau\right\} $, to the density function 
	$$ y\longrightarrow \dfrac{f_{u-t}(y-\zeta^z_{t})f_{r-u}(z-y)}{f_{r-t}(z-\zeta^z_{t})}. $$
	It follows from Scheffé's lemma that
		\[
		\underset{n\rightarrow+\infty}{\lim}\, K_{t_{n},u}(r,\zeta^z_{t_{n}})=K_{t,u}(r,\zeta^z_{t}),
		\] 
		$\mathbb{P}$-a.s on $\left\{ t<\tau\right\} $.
	It is clear from \eqref{phixiconv} and \eqref{varphiconvergence} that $\Upsilon_{n}(r,\zeta_{t_{n}}^{z})$ is bounded and
\begin{equation}
\lim\limits _{n\longrightarrow+\infty}\,\Upsilon_{n}(r,\zeta_{t_{n}}^{z})=\dfrac{\phi(r,t,\zeta_{t}^{z})}{\dint_{(t,+\infty)}\phi(r,t,\zeta_{t}^{z})\mathbb{P}_{\tau}(dr)}.
\end{equation}	
It remains to apply the dominated convergence theorem to deduce \eqref{secondterm}. Thus assertion (i) is proved.
		
	(ii) Let us assume there exists $\varepsilon>0$ such that $\mathbb{P}(\tau>\varepsilon)=1$. Let $(t_{n})_{n\in \mathbb{N}}$
	be a decreasing sequence converging to $0$. Without loss of generality we assume
		$t_{n}< \dfrac{\varepsilon}{2} $ for all $n\in \mathbb{N}$. Then it is sufficient to verify that
		\begin{equation}\label{equationlimitmarkovstep2t=0zeta}
		\lim\limits_{n \rightarrow+\infty} \mathbb{E}[g(\zeta^z_{u})\vert \zeta^z_{t_{n}}]=\mathbb{E}[g(\zeta^z_{u})\vert \zeta^z_{0}],~~\mathbb{P}\text{-a.s.}
		\end{equation}
		Let us recall that 
		\begin{align*}
		\mathbb{E}[g(\zeta^z_{u})\vert\zeta^z_{0}]=\mathbb{E}[g(\zeta^z_{u})]=g(z)F(u)+\int_{(u,+\infty)}\int_{\mathbb{R}}g(y)\dfrac{f_{r-u}(z-y)f_u(y)}{f_r(z)}\,dy\,\mathbb{P}_{\tau}(dr).
		\end{align*}
Now under the above considerations we have  
		\begin{equation*}
	\mathbb{E}[g(\zeta^z_u)|\zeta^z_{t_n}]=\,g(z) \dint_{(\varepsilon,u]}\,\Upsilon_{n}(r,\zeta^z_{t_{n}})\mathbb{P}_{\tau}(dr)+\dint_{(u,+\infty)}\,K_{t_{n},u}(r,\zeta^z_{t_{n}})\Upsilon_{n}(r,\zeta^z_{t_{n}})\,\mathbb{P}_{\tau}(dr).
	\end{equation*}
	 So, in order to show \eqref{equationlimitmarkovstep2t=0zeta} we have to prove that $\mathbb{P}$-a.s, 
		\begin{equation}
		\lim\limits _{n\longrightarrow+\infty}\,\dint_{(\varepsilon,u]}\,\Upsilon_{n}(r,\zeta^z_{t_{n}})\mathbb{P}_{\tau}(dr)=F(u),
		\label{conas1}
		\end{equation}
		and 
		\begin{equation}
		\lim\limits _{n\longrightarrow+\infty}\,\dint_{(u,+\infty)}\,K_{t_{n},u}(r,\zeta^z_{t_{n}})\Upsilon_{n}(r,\zeta^z_{t_{n}})\,\mathbb{P}_{\tau}(dr)=\int_{(u,+\infty)}\int_{\mathbb{R}}g(y)\dfrac{f_{r-u}(z-y)f_u(y)}{f_r(z)}\,dy\,\mathbb{P}_{\tau}(dr).	\label{conas2}
		\end{equation}
		Since the function $(t,r,x)\longrightarrow\dfrac{f_{r-t}(z-x)}{f_r(z)}$ is continuous on $[0, \dfrac{\varepsilon}{2}]\times[\varepsilon, \infty[\times \mathbb{R}$ and, locally uniformly in $x$, 
		$\lim\limits_{r\longrightarrow \infty}\dfrac{f_{r-t}(z-x)}{f_r(z)}=1,$
	then we have 
		\[
\begin{array}{c}
\underset{n\rightarrow+\infty}{\lim}\,\dfrac{f_{r-t_{n}}(z-\zeta^z_{t_{n}})}{f_{r}(z)}=1, \,\,\mathbb{P}\text{-a.s} \\
\\
\underset{n\in\mathbb{N},r>\varepsilon}{\sup}\,\dfrac{f_{r-t_{n}}(z-\zeta^z_{t_{n}})}{f_{r}(z)}<+\infty.
\end{array}
\]
		We conclude assertions \eqref{conas1} and \eqref{conas2} from the Lebesgue dominated convergence theorem.
\end{proof}
We now come to the main result of this section which is the Markov property of $\zeta^z$ with respect to $\mathbb{F}^{\zeta^z,c}_+$ and its consequence, namely the right-continuity of the completed natural filtration of $\zeta^z$.
\begin{theorem}\label{thmrightcontinuityfiltration} 
Under assumption \eqref{behavioursmall} we have

(i) The  process $\zeta^z$ is a Markov process with respect to $\mathbb{F}^{\zeta^z,c}_{+}$.

	(ii) The  filtration $\mathbb{F}^{\zeta^z,c}$ satisfies the usual conditions of right-continuity and completeness. In other words both filtrations $\mathbb{F}^{\zeta^z,c} $ and $\mathbb{F}^{\zeta^z,c}_{+} $ are identical. 
\end{theorem}

\begin{proof}
	(i)	Let us prove that $\zeta^z$ is an $\mathbb{F}_+^{\zeta^z,c}$-Markov process, i.e.,
	\begin{equation}\label{Markovprop}
	\mathbb{E}[g(\zeta^z_{u})\vert\mathcal{F}_{t+}^{\zeta^z}]=\mathbb{E}[g(\zeta^z_{u})\vert\zeta^z_t], \mathbb{P}\text{-a.s.},
	\end{equation}
	for all $0\leq t<u$ and for every bounded measurable function $g$.
		Without loss of
		generality we may assume that the function $g$ is continuous and bounded. Let $(t_{n})_{n\in \mathbb{N}}$
		be a decreasing sequence converging to $t$. Since $g$ is bounded, $\mathcal{F}_{t+}^{\zeta^z,c}=\underset{n}{\bigcap}\mathcal{F}_{t_{n}}^{\zeta^z,c}$ and $\zeta^z$ is an $\mathbb{F}^{\zeta^z}$-Markov process then,  $\mathbb{P}$-a.s., we have
		\begin{align*}
		\mathbb{E}[g(\zeta^z_{u})\vert\mathcal{F}_{t+}^{\zeta^z}]&=\lim\limits_{n\longmapsto+\infty}\mathbb{E}[g(\zeta^z_{u})|\mathcal{F}^{\zeta^z,c}_{t_n}]\nonumber\\
		&=\lim\limits_{n\longmapsto+\infty}\mathbb{E}[g(\zeta^z_{u})\vert\zeta^z_{t_n}].
		\end{align*}
		To prove \eqref{Markovprop} we need to show, for all $t\geq 0$, that 
		\begin{equation}
		\lim\limits_{n\longmapsto+\infty}\mathbb{E}[g(\zeta^z_{u})\vert\zeta_{t_n}]=\mathbb{E}[g(\zeta^z_{u})\vert\zeta^z_{t}], \mathbb{P}\text{-a.s.}\label{eqthetransitionfuncionrigthcontinuous}
		\end{equation}
According to Lemma \ref{lemmatransitionfunctionrigthcontinuous} the only case that remains to be proved is $t=0$ with $\mathbb{P}(\tau>0)=1$. To see this it is sufficient to show that $\mathcal{F}_{0+}^{\zeta^z,c}$ is $\mathbb{P}$-trivial. This amounts to prove that $\mathcal{F}_{0+}^{\zeta^z}$ is $\mathbb{P}$-trivial, since $\mathcal{F}_{0+}^{\zeta^z,c}=\mathcal{F}_{0+}^{\zeta^z} \vee \mathcal{N}_{P}$.

To do so, first of all according to assertion (i) in Definition \ref{defzeta^z} there exists a set $\mathfrak{A}\in \mathcal{B}(0,+\infty)$ such that $\mathbb{P}_{\tau}(\mathfrak{A})=0$ and $0<f_r(z)<\infty$ for all $r\in (0,+\infty)\setminus \mathfrak{A}$. Let $\varepsilon \in (0,+\infty)\setminus \mathfrak{A}$ be fixed and consider the stopping time $\tau_{\varepsilon}=\tau \vee \varepsilon$. Then it is easy to see that $\mathbb{P}_{\tau_{\varepsilon}}(\mathfrak{A})=0$ and therefore we have $0<f_r(z)<\infty$ for $\mathbb{P}_{\tau_{\varepsilon}}$ almost every $r$. Now let $\zeta_{t}^{\tau_{\varepsilon},z}$ be the Lévy bridge from $0$ to $z$ with random length $\tau_{\varepsilon}$.
		First observe that the sets $(\tau_{\varepsilon}>\varepsilon)$ and $(\tau>\varepsilon)$ are equal and the same holds true for$( \tau_{\varepsilon}=r )$ and $( \tau=r)$ for every $r>\varepsilon$. Therefore the following equality of processes holds $$\zeta_{\cdot}^{\tau_{\varepsilon},z}\mathbb{I}_{(\tau>\varepsilon)}=\zeta^z_{\cdot}\;\mathbb{I}_{(\tau>\varepsilon)}.$$
		Then for each $A\in \mathcal{F}_{0+}^{\zeta^z}$ there exists $B\in \mathcal{F}_{0+}^{\zeta^{\tau_{\varepsilon},z}}$ such that  
		$$A\cap(\tau>\varepsilon)=B\cap (\tau>\varepsilon).$$
		As $\mathbb{P}(\tau_{\varepsilon}>\varepsilon/2)=1$ then according to Lemma \ref{lemmatransitionfunctionrigthcontinuous} the process $\zeta_{t}^{\tau_{\varepsilon},z}$ is an $\mathbb{F}_+^{\zeta_{t}^{\tau_{\varepsilon},z},c}$-Markov process. Thus $\mathcal{F}_{0+}^{\zeta^{\tau_{\varepsilon}},z}$ is $\mathbb{P}$-trivial which implies that $\mathbb{P}(B)=0$ or $1$. Consequently we obtain 
		$$\mathbb{P}(A\cap(\tau>\varepsilon))=0\text{\,\,or\,\,}\mathbb{P}(A\cap(\tau>\varepsilon))=\mathbb{P}(\tau>\varepsilon).$$
		Now if $\mathbb{P}(A)>0$, then there exists $\varepsilon \in (0,+\infty)\setminus \mathfrak{A}$ such that $\mathbb{P}(A\cap \{\tau> \varepsilon\})>0$. Therefore for all $0<\varepsilon'\leq \varepsilon$ we have 
		$$\mathbb{P}(A\cap(\tau>\varepsilon'))=\mathbb{P}(\tau>\varepsilon').$$ 
		Passing to the limit as $\varepsilon'$ goes to $ 0$ yields $\mathbb{P}(A\cap(\tau>0))=\mathbb{P}(\tau>0)=1$. It follows that $\mathbb{P}(A)=1$.

(ii) It is sufficient to prove that for every bounded
	$\mathcal{F}_{t+}^{\zeta^z,c}$-measurable $Y$ we have $\mathbb{P}$-a.s.,
	\begin{equation}
	\mathbb{E}[Y|\mathcal{F}_{t+}^{\zeta^z,c}]=\mathbb{E}[Y|\mathcal{F}_{t}^{\zeta^z,c}].\label{eqrightcontinuityfiltration}
	\end{equation}
	This is a consequence of the Markov property of $\zeta^z$ with respect to $\mathbb{F}^{\zeta^z,c}_+$. 
		\end{proof}
		\section{Examples}
		In this section we shall give some examples of Lévy bridges with random length.
		
\begin{enumerate}
\item \textbf{~Stable process}.

Symmetric $\alpha$-stable processes $X^{S}_{\alpha}$, with $\alpha \in (0,1)\cup(1,2)$, are the class of L{\'e}vy processes whose characteristic exponents 
correspond to those of symmetric $\alpha$-stable distributions. The corresponding L{\'e}vy measure is given by
\[
d\Lambda^{S}_{\alpha}(s)=\left(\frac{1}{\left|s\right|^{1+\alpha}}1\!\!1_{\{s<0\}}+
\frac{1}{s^{1+\alpha}}1\!\!1_{\{s>0\}}\right)ds.\]
The characteristic exponent $\varPsi_{\Lambda^{S}_{\alpha}}$ has the form
\[
\varPsi_{\Lambda^{S}_{\alpha}}(u)=-|u|^{\alpha},~ u\in\mathbb{R}.
\]

\item \textbf{~Tempered stable process}. 

Tempered stable processes belong to the class proposed by Rosinski (2007). They are Lévy processes with no Gaussian component and Lévy density of the form:

\[
d\Lambda^{TS}_{\alpha}(s)=\left(\frac{c_{-}e^{-\lambda_{-}\left|s\right|}}{\left|s\right|^{1+\alpha}}1\!\!1_{\{s<0\}}+\frac{c_{+}e^{-\lambda_{+} s}}{s^{1+\alpha}}1\!\!1_{\{s>0\}}\right)ds\]
with parameters satisfy $c_{-} >0,\, c_{+} >0,\, \lambda_{-} >0,\, \lambda_{+} >0$,  and $0< \alpha <2$.
In particular, the class considered if $c_{+} = c_{-}=c$ and $\lambda_{+} = \lambda_{-}=\lambda$. The characteristic exponent $\varPsi_{\Lambda^{TS}_{\alpha}}$ is given by
\[
\varPsi_{\Lambda^{TS}_{\alpha}}(u)=\Gamma(-\alpha)\,c\, \lambda^{\alpha}\left[ \left(1-i\dfrac{u}{\lambda}\right)^{\alpha}+\left(1+i\dfrac{u}{\lambda}\right)^{\alpha}-2\right],~ u\in\mathbb{R}.
\]

For the stable and tempered stable processes the conditions \eqref{hyintegr} and \eqref{behavioursmall} are satisfied, see Remark 6.4 and Example 6.5 in \cite{FH}.

\item \textbf{~Modified tempered stable process}. 

The modified tempered stable process is a L{\'e}vy process 
with the L{\'e}vy density given by 
\[
d\Lambda^{MTS}_{\alpha}(s)=\frac{1}{\pi}\left(\frac{K_{\alpha+\frac{1}{2}}(\left|s\right|)}{\left|s\right|^{\alpha+\frac{1}{2}}}1\!\!1_{\{s<0\}}+\frac{K_{\alpha+\frac{1}{2}}(s)}{s^{\alpha+\frac{1}{2}}}1\!\!1_{\{s>0\}}\right)ds.
\]
 Here $K_{\alpha+\frac{1}{2}}$ is the modified Bessel function of the second kind given by the following integral representation
\begin{equation*}
K_{\alpha+\frac{1}{2}}(s)=
\frac{1}{2}\left(\frac{s}{2}\right)^{\alpha+\frac{1}{2}}\int_{0}^{+\infty}
\exp\left( -t- \frac{s^{2}}{4t}\right) t^{-\alpha-\frac{3}{2}}\, dt.\label{Aout26eq2} 
\end{equation*}
The characteristic exponent has the form 
\[
\varPsi_{\Lambda^{MTS}_{\alpha}}(u)=\frac{1}{\sqrt{\pi}}\,2^{-\alpha-\frac{1}{2}}\,\Gamma(-\alpha)[(1+u^{2})^{\alpha}-1], ~u\in\mathbb{R}.
\]
For additional details the reader may consult \cite{KRCB}. 

\item\textbf{~Normal inverse Gaussian process}. 

The Normal inverse Gaussian process is a L{\'e}vy process with L{\'e}vy measure given by
\[
 d\Lambda^{NIG}(s)=\frac{K_{1}(|s|)}{\pi |s|}ds,
\]
where $K_{1}$ is the modified Bessel function of the second
kind with index $1$. The characteristic exponent is equal to 
\[
\varPsi_{\Lambda^{\mathrm{NIG}}}(u)= \left(1-\sqrt{1+u^{2}}\right), u\in \mathbb{R}.
\]
For further results related to the Normal inverse Gaussian process see
\cite{Br2}  and \cite{Ryd1}.
\end{enumerate}

One can easily check that the conditions \eqref{hyintegr} and \eqref{behavioursmall} are satisfied for both previous processes.

\vspace{0,2cm}

\textbf{Acknowldgements:} 		The authors would like to express particular thanks to Hans-J\"urgen Engelbert for comments that greatly improved the manuscript and René Schilling for clarification of a result in his paper.
The third author gratefully acknowledge financial support by an Erasmus+ International Credit Mobility exchange project coordinated by Linnaeus University. This article has been finalizedduring a staf mobility of the second author at Cadi Ayyad University within this project.


\begin{thebibliography}{99}
	
	\bibitem{Br2}
Barndorff-Nielsen, O.~E. Processes of normal inverse Gaussian type.
Finance Stoch., 2, 41-68, (1998).
		\bibitem{B}
		Bertoin J. \textit{L\'evy Processes}. Cambridge Univeristy Press, Cambridge, (1996).
	\bibitem{BBE}
	Bedini, M. L.; Buckdahn, R.; Engelbert, H. J. Brownian bridges on random intervals. Theory Probab. Appl., 61, 15-39, (2017).
	\bibitem{BG}
 Blumenthal, R. M.; Getoor, R. K. \textit{ Markov Processes and Potential Theory}. Academic Press, New York-London (1968).
  \bibitem{CUB}
 Chaumont, L.; Uribe-Bravo, G. Markovian bridges:
 Weak continuity and pathwise constructions. Ann. Probab., 39, 609-647, (2011).
  \bibitem{EL}
 Erraoui, M.; Louriki, M. Bridges with random length: Gaussian-Markovian case. Markov Process. Related Fields.,  4, 669-693, (2018).
  \bibitem{EHL}
  Erraoui, M.; Hilbert, A.; Louriki, M. Bridges with random length: gamma case. J. Theoret. Probab. resubmitted after minor revisions, (2019).
 \bibitem{EY}
 Emery, M.; Yor, M. A parallel between Brownian bridges and gamma bridges.
 Publ. RIMS, Kyoto Univ., 40, 669-688, (2004).
 \bibitem{FG}
 Fitzsimmons, P. J.; Getoor, R. K. Occupation time distributions for L\'evy
 bridges and excursions. Stochastic Process. Appl., 58, 73-89, (1995)
 	\bibitem{FPY}
 	Fitzsimmons P.J.;  Pitman J.;  Yor M. Markovian bridges: Construction, Palm interpretation, and splicing. Seminar on Stochastic Processes, 33, 102-133, (1993).
 	\bibitem{FH}
 	Figueroa-López, J. E.; Houdré, C. Small-time expansions for the transition distributions of Lévy processes. Stochastic Process. Appl., 119, 3862-3889, (2009).
 	\bibitem{GSV}
Gasbarra, D.; Sottinen, T.; Valkeila, E. Gaussian bridges. Stochastic Analysis and Applications,  Abel Symp., 2, Springer, Berlin, 361-382, (2007).
\bibitem{GS}
Getoor, R. K.; Sharpe, M. J. Excursions of dual processes. Adv. in Math., 45, 259-309, (1982).
 \bibitem{HHM}
  Hoyle, E.;  Hughston, L.P.;  Macrina, A. L\'evy random bridges and the modelling
 of financial information. Stochastic Process. Appl., 121, 856-884, (2011).
 \bibitem{HW}
 Hartman, P; Wintner, A. On the infinitesimal generators of integral convolutions,
 Amer. J. Math., 64, 273-298,  (1942).
 \bibitem{Ish}
 Ishikawa, Y. Asymptotic behavior of the transition density for jump
type processes in small time. Tohoku Math. J., 46,
443-456, (1994).
\bibitem{Kal73}
Kallenberg, O. Canonical representations and convergence criteria for processes with interchangeable increments, Z. Wahrscheinlichkeitstheorie  Verw. Geb. 27, 23-36, (1973). 
\bibitem{Kal81}
Kallenberg, O. Splitting at backward times in regenerative sets. Ann. Probab., 9, 781-799,  (1981).
 \bibitem{KS}
 Knopova, V.; Schilling, R.L.  A note on the existence of
 transition probability densities of L\'evy processes, Forum Math., 25, 125-149, (2013).
 \bibitem{Kol}
 Kolokoltsov, V.N. \textit{Semiclassical Analysis for Diffusions and Stochastic Processes}, Lecture Notes in Mathematics 1724, Springer-Verlag, (2000).
 \bibitem{Lea87}
 Léandre, R. Densité en temps petit d'un processus de sauts. Séminaire
de Probabilités, XXI, 81-99, Lecture Notes in Math., 1247, Springer,
Berlin, (1987).

	
\bibitem{KRCB}
Rachev, Y.~S., Chung, S.~T., Kim, D.~M.; Bianchi, M.~L. The modified tempered stable distribution, {GARCH} models and option
  pricing. Probab. Math. Statist., 29, 91-117, (2009).
 \bibitem{Ryd1}
 Rydberg, T.~H. The normal inverse {G}aussian {L}{\'e}vy process: simulation and
  approximation. Comm.~Statist.~Stochastic Models, 13, 887-910, (1997).  	  
			\bibitem{Sato}
			Sato, K. I. \textit{L\'evy Processes and Infinitely Divisible Distributions}. Cambridge University Press,
			Cambridge, (1999).
			\bibitem{Sc}
			Schenk, W. Über das asymptotische Verhalten der Übergangswahrscheinlichkeiten eines stochastischen Prozesses mit stationären und unabhängigen Zuwächsen. (German) Wiss. Z. Techn. Univ. Dresden 24, 945-947, (1975).
		\bibitem{Sh}
		Sharpe, M. Zeroes of infinitely divisible densities. Ann. Math. Statist., 40, 1503-1505. (1969).
	
	\bibitem{T}
Tucker, H.G. Absolute continuity of infinitely divisible distributions, Pac. J. Math., 12  1125-1129, (1962).
\bibitem{W}
Wittmann, R. Natural densities of Markov transition probabilities. Probab. Theory Relat. Fields., 73, 1-10, (1986).
\end{thebibliography}
\end{document}